\documentclass[a4paper,12pt]{amsart}

\usepackage{graphicx}
\usepackage{amsmath}
\usepackage{amssymb}
\usepackage{amscd}

\usepackage{amsmath,amsfonts,amssymb,eufrak}
\usepackage{graphics}
\usepackage{latexsym}
\usepackage{xy}
\usepackage{rotating}

\input xy
\xyoption{all}

\newcommand{\comment}[1]{}

\newcommand{\ul}[1]{\underline{#1}}
\newcommand{\ol}[1]{\overline{#1}}

\renewcommand{\mod}{\operatorname{mod}\nolimits}
\newcommand{\add}{\operatorname{add}\nolimits}
\newcommand{\rad}{\operatorname{rad}\nolimits}
\newcommand{\soc}{\operatorname{soc}\nolimits}
\newcommand{\ind}{\operatorname{ind}\nolimits}

\DeclareMathOperator{\CM}{CM}
\newcommand{\op}{{\operatorname{op}\nolimits}}
\newcommand{\coh}{\operatorname{coh}\nolimits}

\newcommand{\Fac}{\operatorname{Fac}\nolimits}
\newcommand{\Sub}{\operatorname{Sub}\nolimits}

\newcommand{\End}{\operatorname{End}\nolimits}

\newcommand{\rep}{\operatorname{rep}\nolimits}
\newcommand{\Ext}{\operatorname{Ext}\nolimits}
\newcommand{\Hom}{\operatorname{Hom}\nolimits}
\newcommand{\CC}{\mathcal{C}_{H}}

\newcommand{\Z}{\mathbb{Z}}
\DeclareMathOperator{\uCM}{\underline{CM}} \DeclareMathOperator{\uHom}{\underline{Hom}}

\newcommand{\DBH}{{\bf D}^b(H)}

\newcommand{\DIV}{\mathbf{D}_4}
\newcommand{\DV}{\mathbf{D}_5}
\newcommand{\modL}{\textrm{mod} \Lambda}
\newcommand{\stmodL}{\underline{\textrm{mod}} \Lambda}

\newcommand{\modH}{\textrm{mod} H}
\newcommand{\stSubG}{\underline{\Sub}\Gamma}
\newcommand{\stSubL}{\underline{\Sub}\Lambda}
\newcommand{\modG}{\textrm{mod} \Gamma}

\newcommand{\secI}{Section~1 }
\newcommand{\secII}{Section~2 }
\newcommand{\secIII}{Section~3 }
\newcommand{\secIV}{Section~4 }
\newcommand{\secV}{Section~5 }

\newcommand{\citeRV}{[RV]}

\newcommand{\citeAR}{[AR,3.2,2.1]}

\newcommand{\citeI}{[I]}

\renewcommand{\dim}{\operatorname{dim}\nolimits}

\newcommand{\Id}{\operatorname{id}\nolimits}

\newcommand{\la}{\Lambda}

\newcommand{\xto}{\xrightarrow}

\newtheorem{theorem}{Theorem}[section]

\newtheorem{corollary}[theorem]{Corollary}

\newtheorem{lemma}[theorem]{Lemma}

\newtheorem{proposition}[theorem]{Proposition}

\newcommand{\cc}{{\mathcal C}}

\newcommand{\cF}{{\mathcal F}}

\newcommand{\ci}{{\mathcal I}}

\newcommand{\cp}{{\mathcal P}}

\newcommand{\ct}{{\mathcal T}}

\newcommand{\cx}{{\mathcal X}}
\newcommand{\cy}{{\mathcal Y}}

\newcommand{\C}{\mathcal{C}}
\newcommand{\CH}{{\C_H}}

\begin{document}
\title{Tilting theory and cluster algebras}
\author{Idun Reiten}
\maketitle
\section*{Introduction}
The purpose of this chapter is to give an introduction to the theory of cluster
categories and cluster-tilted algebras, with some background on the theory of cluster
algebras, which motivated these topics. We will also discuss some of the interplay
between cluster algebras on one side and cluster categories/cluster-tilted algebras on
the other, as well as feedback from the latter theory to cluster algebras.

The theory of cluster algebras was initiated by Fomin and Zelevinsky \cite{fz1}, and
further developed by them in a series of papers, including \cite{fz2}, some involving
other coauthors. This theory has in recent years had a large impact on the representation
theory of algebras. The first connection with quiver representations was given in
\cite{mrz}. Then the cluster categories were introduced in \cite{bmrrt} in order to model
some of the ingredients in the definition of a cluster algebra. For this purpose a
tilting theory was developed in the cluster category. (See \cite{ccs1} for the
independent construction of a category in the $A_n$ case which turned out to be
equivalent to the cluster category \cite{ccs2}). This led to the theory of cluster-tilted
algebras initiated in \cite{bmr1} and further developed in many papers by various
authors.

The theory of cluster-tilted algebras (and cluster categories) is closely connected with
ordinary tilting theory. Much of the inspiration comes from usual tilting theory. Features missing in tilting theory when trying to model clusters from the theory of
cluster algebras made it necessary to replace the category $\mod H$ of finitely generated
$H$-modules for a finite dimensional hereditary algebra $H$ with a related category which
is now known as the cluster category. On the other hand, the theory of cluster-tilted
algebras provides a new point of view on the old tilting theory.

The Bernstein-Gelfand-Ponomarev (BGP) reflection functors were an important source of
inspiration for the development of tilting theory, which provided a major generalization
of the work in \cite{bgp}. The Fomin-Zelevinsky (FZ) mutation, which is an essential
ingredient in the definition of cluster algebras, gives a generalization of these
reflections in another direction.

We start with introducing cluster algebras in the first section. We illustrate the
essential concepts with an example, which will be used throughout the chapter. We give
main results and conjectures about cluster algebras which are relevant for our further
discussion. In \secII  we introduce and investigate cluster categories, followed by
cluster-tilted algebras in \secIII. In \secIV we discuss the interplay between cluster
algebras and cluster categories/cluster-tilted algebras, and we also give applications to
cluster algebras. The cluster categories are a special case of the more general class of
$\Hom$-finite triangulated Calabi-Yau categories of dimension 2 (2-CY categories), and
much of the theory generalizes to this setting. Denote by $\mod\la$ the  category of finitely generated (left) modules over a finite dimensional algebra $\la$. An important case is the stable category
$\ul{\mod}\la$, where $\la$ is the preprojective algebra  of a Dynkin quiver. The related category $\mod\la$ has been studied extensively by Geiss-Leclerc-Schr\"oer, who
extended results from cluster categories to this setting, and gave applications to
cluster algebras \cite{gls1}\cite{gls2}\cite{gls3}. See also \cite{birsc}.  We treat this in \secV.

We suppose that the reader is familiar  with the basic theory of quiver representations
and almost split sequences (see \cite{rin1},\cite{ars},\cite{ass} and other chapters in
this volume). We also presuppose some background  from ordinary tilting theory (see
\cite{rin1},\cite{ass},\cite{ahk}), but here we shall nevertheless recall relevant
definitions and results when they are needed. We generally do not give proofs, but
sometimes we include some indication of proofs in order to stress some ideas. Instead we
give examples to illustrate the theory, and we try to give some motivation for the work.
We should also emphasize that the selection of the material reflects our personal
interests.

 For each section we add some historical notes with references at the end, rather
than giving too many references as we go along. We also refer to the surveys
\cite{bm1}\cite{rin2}\cite{ke2}\cite{ke3}. We assume throughout that we work over a field $k$ which is         
algebraically closed.
    
This chapter is based on the series of lectures I gave in Trieste in January 2006. I
would like to thank I. Muchtadi Alumsayh and G. Bobinski for providing me with a copy of
their notes, and Aslak Bakke Buan, Osamu Iyama, Bernhard Keller and David Smith for their very helpful
comments.

\section{Cluster algebras}
In this section we introduce a special class of cluster algebras and illustrate the
underlying concepts through a concrete example. We also give a selection of main results
and conjectures of Fomin-Zelevinsky which provide an appropriate background for our
further discussion.

\bigskip
\subsection{Fomin-Zelevinsky mutation}
Let $Q$ be a finite connected quiver with vertices $1,2,\cdots, n$. We say that $Q$ is a
\emph{cluster quiver} if it has no loops $\xymatrix@C0.2cm@R0.2cm{ \bullet\ar@(ur,dr)[]&
}$
 and no 2-cycles
$\xymatrix@C0.2cm@R0.2cm{
   \bullet \ar@/^/[rr]&&\bullet\ar@/^/[ll]&
}$ For each vertex $i=1,2,\cdots,n$, we define a new quiver $\mu_i(Q)$ obtained by
\emph{mutating} $Q$, and we call the process Fomin-Zelevinsky mutation, or FZ-mutation
for short.

The quiver $\mu_i(Q)$ is obtained from $Q$ as follows.

\noindent (i) Reverse all arrows starting or ending at $i$.

\noindent (ii) If in $Q$ we have $n>0$ arrows from $t$ to $i$ and $m>0$ arrows from $i$
to $s$ and $r$ arrows from $t$ to $s$ (interpreted as $-r$ arrows from $s$ to $t$ if
$r<0$), then in the new quiver $\mu_i(Q)$ we have $nm-r$ arrows from $s$ to $t$
(interpreted as $r-nm$ arrows from $s$ to $t$ if $nm-r<0$).

An important and easily verified property of the mutation is the following.
\begin{proposition}
For a cluster quiver $Q$, we have $\mu_i(\mu_i(Q))=Q$  for each vertex $i$ of $Q$.
\end{proposition}

We illustrate with some examples.

\noindent {\bf (a)} Let $Q$ be the quiver

$$
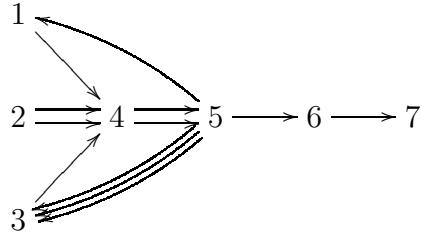
\xymatrix{
  1\ar[dr]&&&&\\
  2\ar@<0.5ex>[r]\ar@<-0.5ex>[r]& 4\ar@<0.5ex>[r]\ar@<-0.5ex>[r]&
5\ar[r]\ar@/^/@<0.5ex>[dll]\ar@/^/@<-0.5ex>[dll]\ar@/^/[dll]\ar@/_/[ull]&6\ar[r]&7\\
  3\ar[ur]&&&&
}$$



Then $\mu_4(Q)$ is the quiver

$$
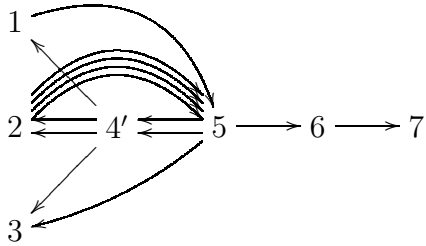
\xymatrix{
1\ar@/^2pc/[drr]&&&&\\
2\ar@/^2pc/@<0.3ex>[rr]\ar@/^2pc/@<-0.3ex>[rr]\ar@/^2pc/@<0.9ex>[rr]\ar@/^2pc/@<-0.9ex>[rr]&
4'\ar@<0.5ex>[l]\ar@<-0.5ex>[l]\ar[dl]\ar[ul]&
5\ar@<0.5ex>[l]\ar@<-0.5ex>[l]\ar@/^/[dll]\ar[r]& 6\ar[r]& 7\\
3&&&& }$$ 
For example there are $2\cdot 2=4$ arrows from vertex $2$ to vertex $5$ in $\mu_4(Q)$, and
considering the paths between $3$ and $5$, we have $1\cdot 2-3=-1$, so that there is one
arrow from $5$ to $3$.

\noindent {\bf (b)} Let $Q$ be the quiver $1\to 2\to 3$. Then  $Q'=\mu_3(Q) $ is the
quiver $1\to 2\leftarrow 3$ obtained by reversing the arrows involving 3. Since in this case 3 is a sink in the quiver, there is no path of length two with middle vertex
3. The same thing happens when we mutate at a vertex which is a source.

\bigskip
Hence we see that when we mutate at a sink or a source, the procedure coincides with the
BGP-reflections.

When we have a BGP-reflection, like in the above example, there is an equivalence between
the subcategories of the categories of finite dimensional representations $\rep Q$ and
$\rep Q'$ obtained by ``removing'' in each case the simple representation at the vertex 3
\cite{bgp}.

\bigskip
\subsection{Definition of cluster algebras}

Let $Q$ be a cluster quiver with vertices $1,2,\cdots, \\n$ and let
$F=\mathbb{Q}(x_1,\cdots,x_n)$ be the function field in $n$ indeterminates over
$\mathbb{Q}$. Consider the pair $(\ul{x},Q)$, where $\ul{x}=\{x_1,\cdots,x_n\}$. The
cluster algebra $C(\ul{x},Q)$ will be defined to be a subring of $F$. The main
ingredients involved in the definition are the following concepts: \emph{cluster},
\emph{cluster variable}, \emph{seed, mutation of seeds}.

The pair $(\ul{x},Q)$ consisting of a free generating set $\ul{x}$ for $F$ over the
rational numbers $\mathbb{Q}$, together with a quiver with $n$ vertices, is called a
\emph{seed}. Here the $n$ elements in $\ul x$ are viewed as labeling the vertices of the quiver $Q$. We consider the elements in $\ul x$ as an ordered set, with corresponding ordering of the vertices of the quiver, written from left to right. If $(\ul x', Q')$ is obtained from $(\ul x, Q)$ by simultaneous rearrangement of the elements in $\ul x$ and the vertices in $Q$, then $(\ul x', Q')$ is a seed \textit{equivalent} to $(\ul x, Q)$ and we will identify $(\ul x, Q)$ with $(\ul x', Q')$. 

For $i=1,\cdots,n$ we define a mutation $\mu_i$ taking the seed $(\ul{x},Q)$
to a new seed $(\ul{x}',Q')$, where $Q'=\mu_i(Q)$ as discussed in 1.1, and $\ul{x}'$ is
obtained from $\ul{x}$ by replacing $x_i$  by a new element $x_i'$ in $F$. Here $x_i'$ is
defined by $x_i x_i'=m_1+m_2$,  where $m_1$ is a monic monomial in the variables
$x_1,\cdots,x_n$, where the power of  $x_j$ is the number of arrows from $j$ to $i$, and
$m_2$ is the monic monomial where the power of $x_j$  is the number of arrows from $i$ to
$j$. If there is no arrow from $j$ to $i$, then $m_1=1$, and if there is no arrow from
$i$ to $j$, then $m_2=1$. Note that while in the new seed the quiver $Q'$ only depends
on the quiver $Q$, the set $\ul{x}' $ depends on both $\ul{x}$ and $Q$. We have
$\mu_i^2(\ul{x},Q)=(\ul{x},Q)$.

We perform this operation for all $i=1,\cdots,n$, then perform it on the new seeds etc.
(or we get back to one of the seeds equivalent to one already computed). This gives rise to a graph which may be finite or infinite. The $n$-element subsets
$\ul{x},\ul{x}',\ul{x}'',\cdots$ occurring  are by definition the \emph{clusters}, the
elements in the clusters are the \emph{cluster variables}, and the \emph{seeds} are all
pairs $(\ul{x}',Q')$ occurring. The corresponding \emph{cluster algebra} $C(\ul{x},Q)$,
which as an algebra only depends on $Q$, is the subring of $F$ generated by the cluster
variables. We also write $C(\ul{x},Q)=C(Q)$.

When we are given the cluster algebra only, the information on the clusters, cluster
variables and seeds may be lost, and also the rule for mutation of seeds. We want to keep
all this information in mind,  in addition  to the cluster algebra itself, which is
determined by this information.

We remark that the more general definition of cluster algebras includes the possibility
of having so-called coefficients, and it also allows valued quivers. In the language of
\cite{fz1} the last generalization means to consider skew symmetrizable matrices rather
than just skew symmetric ones. The correspondence between quivers and matrices is
illustrated by the following example: The quiver $Q\colon 1\to 2\to 3$ corresponds to the
matrix $\left(
\begin{smallmatrix}
0&1&0\\
-1&0&1\\
0&-1&0
\end{smallmatrix}
\right)$. The arrow $1\to 2$ gives rise to the entries $a_{12}=1$ and $a_{21}=-1$ and the arrow $2\to 3$ to the entries $a_{23}=1$ and $a_{32}=-1$. Since there are no more arrows, the remaining entries are zero. 

We say that two quivers are \textit{mutation equivalent} if one can be obtained from the other using a finite number of mutations. If $Q'$ is a quiver which is mutation equivalent to $Q$, then the cluster algebras $C(Q') $ and $C(Q)$ are isomorphic.

\subsection{An example}
Let $Q_1$ be  the quiver $\xymatrix{\bullet\ar[r]&\bullet\ar[r]&\bullet}$. The mutation class of $Q_1$ has in addition the quivers  $Q_2=\xymatrix{\bullet&\bullet\ar[l]\ar[r]&\bullet}$,
$Q_3=\xymatrix{\bullet\ar@/_/[rr]&\bullet\ar[l]&\bullet\ar[l]}$ and 
$Q_4=\xymatrix{\bullet\ar[r]&\bullet&\bullet\ar[l]}$. Let $\ul{x}=\{x_1,x_2,x_3\}$, where $x_1,x_2,x_3$
are indeterminates, and $F=\mathbb{Q}(x_1,x_2,x_3)$. By performing the three mutations of the seed $(\ul x, Q_1)$ we get $\mu_1(\ul{x},Q_1)=(\ul{x}',Q_2)$, where $\ul{x}'=\{x_1',x_2,x_3\}$, with $x_1x_1'=1+x_2$, so that $x_1'=\frac{1+x_2}{x_1}$, $\mu_2(\ul{x},Q_1)=(\ul{x}'',Q_3)$, where $\ul{x}''=\{x_1,x_2'',x_3\}$, where $x_2x_2''=x_1+x_3$, so that
$x_2''=\frac{x_1+x_3}{x_2}$. 

Continuing this process,we get the graph shown in Figure \ref{clusterExFig}, called the \textit{cluster graph}.  The clusters are: \vspace{2mm}$\{ x_1,x_2,x_3 \}$, $\{
\frac{1+x_2}{x_1},x_2,x_3 \}$, $\{ \ x_1, \frac{x_1 + x_3}{x_2},x_3 \}$, $\{ x_1,x_2,
\frac{1+x_2}{x_3}\}$, $\{ \frac{1+x_2}{x_1}, \frac{x_1 +
  (1+x_2)x_3}{x_1 x_2},x_3 \}$,\vspace{2mm}
$\{\frac{1+x_2}{x_1},x_2,\frac{1+x_2}{x_3} \}$, $\{\frac{x_1+(1+x_2)x_3}{x_1
x_2},\frac{x_1 + x_3}{x_2},x_3 \}$, $\{x_1,\frac{x_1+x_3}{x_2},\frac{(1+x_2)x_1+x_3}{x_2
x_3} \}$, $\{x_1,\frac{(1+x_2)x_1 + x_3}{x_2 x_3},\frac{1+x_2}{x_3} \}$,\vspace{2mm}
$\{\frac{1+x_2}{x_1},\frac{x_1+(1+x_2)x_3}{x_1 x_2},\frac{(1+x_2)x_1 +(1+x_2)x_3}{x_1 x_2
x_3}\}$, $\{\frac{1+x_2}{x_1},\frac{(1+x_2)x_1 +(1+x_2)x_3}{x_1 x_2
x_3},\frac{1+x_2}{x_3} \}$, $\{\frac{x_1+(1+x_2)x_3}{x_1
x_2},\frac{x_1+x_3}{x_2},\frac{(1+x_2)x_1
+(1+x_2)x_3}{x_1 x_2 x_3} \}$\vspace{2mm},\\
$\{\frac{(1+x_2)x_1 +(1+x_2)x_3}{x_1 x_2
x_3},\frac{x_1+x_3}{x_2},\frac{(1+x_2)x_1+x_3}{x_2 x_3} \}$, $\{\frac{(1+x_2)x_1
+(1+x_2)x_3}{x_1 x_2 x_3},\frac{(1+x_2)x_1+x_3}{x_2 x_3},\frac{1+x_2}{x_3} \}$, and the
\vspace{2mm}cluster variables are: $x_1$, $x_2$, $x_3$, $\frac{1+x_2}{x_1}$,
$\frac{x_1+x_3}{x_2}$, $\frac{1+x_2}{x_3}$, $\frac{x_1+(1+x_2)x_3}{x_1 x_2}$,
$\frac{(1+x_2)x_1 + x_3}{x_2 x_3}$, $\frac{(1+x_2)x_1 + (1+x_2)x_3}{x_1 x_2 x_3}$.
\tiny
\begin{sidewaysfigure}[p]
\[
\xymatrix@C-5.5pc@!C{
&&&&&& \\
&&&&&& \\
&&&&&& \\
&&&&&& \\
&&&&&& \\
&&&&&& \\
&&&&&& \\
&&&&&& \\
&&&(\{x_1,x_2,x_3\},Q_{1})\ar@{-}[dll]^{\mu_1}\ar@{-}[dd]^{\mu_2}\ar@{-}[drr]^{\mu_3}&&&\\
&(\{\frac{1+x_2}{x_1},x_2,x_3\},Q_2)\ar@{-}[ldd]^{\mu_2}\ar@{--}[rrddd]&&
&&
(\{x_1,x_2,\frac{1+x_2}{x_3}
\},Q_4)\ar@{-}[ddr]^{\mu_2}\ar@{--}[llddd]& \\
&&&(\{x_1,\frac{x_1+x_3}{x_2},x_3 \},Q_3)\ar@{--}[ld]\ar@{--}[rd]&&&\\
(\{\frac{1+x_2}{x_1},\frac{x_1 + (1+x_2)x_3}{x_1
  x_2},x_3\},Q_4)\ar@{-}[rddd]^{\mu_3}\ar@{--}[rr]&&(\{x_3, \frac{x_1 + (1+x_2)x_3}{x_1
x_2},\frac{x_1 +
  x_3}{x_2},\},Q_1)\ar@{-}[dddd]^{\mu_1}&&(\{\frac{x_1 +
x_3}{x_2},\frac{(1+x_2)x_1 +
  x_3}{x_2
x_3}, x_1\},Q_1)\ar@{-}[dddd]^{\mu_3}\ar@{--}[rr]&&(\{x_1,\frac{(1+x_2)x_1+x_3}{x_2
  x_3},\frac{1+x_2}{x_3}\},Q_2)\ar@{-}[lddd]^{\mu_1}\\
  &&&(\{\frac{1+x_2}{x_3},x_2,\frac{1+x_2}{x_1}\},Q_1)\ar@{--}[d]&&&\\
  &&& (\{\frac{1+x_2}{x_1},\frac{1+x_2}{x_3},y\},Q_3)\ar@{-}[lld]^{\mu_2}\ar@{--}[rrd]&&&\\
&(\{\frac{1+x_2}{x_1},\frac{x_1 + (1+x_2)x_3}{x_1
  x_2},y\},Q_1)\ar@{--}[dr]&&&&(\{y,\frac{(1+x_2)x_1+x_3}{x_2
x_3},\frac{1+x_2}{x_3}\},Q_1)\ar@{--}[ld]&\\
&&(\{y, \frac{x_1+(1+x_2)x_3}{x_1
x_2},\frac{x_1+x_3}{x_2}\},Q_2)\ar@{--}[rr]&&(\{\frac{x_1 +
x_3}{x_2},\frac{(1+x_2)x_1 + x_3}{x_2 x_3},y\},Q_4)&& } \]
\caption{\label{clusterExFig}Example.}
\end{sidewaysfigure}

\normalsize

Note that for example the seed $(\{x_1,x_2,x_3\}, Q_1)$ is identified with the seed $(\{x_2,x_1,x_3\},\\ \xymatrix{\bullet\ar@/_/[rr]&\bullet\ar[l]&\bullet})$. In Figure 1 we always choose a seed in the equivalence class containing one of the quivers $Q_1,Q_2,Q_3,Q_4$. This means that we sometimes must replace a seed with one which is different from the one directly obtained by mutation. We indicate this by using dotted edges instead of solid edges. For simplicity we have written $y=\frac{(1+x_2)x_1+(1+x_2)x_3}{x_1x_2x_3}$.


\bigskip
\subsection{Some main results}
There are many interesting results in the theory of cluster algebras. Here we
give some of the main theorems and open problems which are of special interest for this chapter. We mainly deal with the acyclic case. 
\vspace{0.3cm} 

\noindent \textbf{(a) Finiteness conditions.} The cluster algebra
$C(Q)=C(\ul{x},Q)$ is said to be of \emph{finite type} if there is only a finite number
of cluster variables. This is equivalent to saying that there is only a finite number of
clusters, and also to the fact that there is only a finite number of seeds. But as we
shall see later, it is not equivalent to having only a finite number of quivers. There is
the following description of finite type.
\begin{theorem}
  Let $Q$ be a cluster quiver. Then the cluster algebra $C(Q)$ is of finite type if
and only if  $Q$ is mutation equivalent to a Dynkin quiver.
\end{theorem}
Note that this result is similar to Gabriel's classification theorem of the quivers of finite representation type, even though the mutation procedure is more complicated than reflections at sinks or sources.

In \secIV we consider the following problem posed by Seven:

\vspace{0.3cm} \noindent \textbf{Problem:} For which quivers $Q$ is the mutation class of $Q$ finite?

\vspace{0.3cm} \noindent \textbf{(b) Laurent phenomenon.} Observe that in the example in
Section 1.3 we see that all denominators of the cluster variables (when written in reduced form)
are monomials. Surprisingly, this is a special case of the following general result.
\begin{theorem}
  Let $C(Q)$ be a cluster algebra with initial seed $(\ul{x},Q)$. Then for any
cluster variable in reduced form $f/g$ (that is, $f$ and $g$ have no common nontrivial factor),  the denominator is a monomial in $x_1,\cdots,x_n$.
\end{theorem}

\vspace{0.3cm} \noindent \textbf{(c) The monomial in the denominators of cluster
variables.} Taking a closer look at the monomials in the denominators in the example in
1.3, we see that interpreting the factors $x_i$ as the simple modules $S_i$ corresponding
to vertex $i$, the denominators correspond to indecomposable modules via the composition
factors. This was already proved in \cite{fz2} for the case of a Dynkin quiver with no
paths of length strictly greater than one. As we shall see later, there are more general
results in this direction, obtained as application of the theory of cluster categories
and cluster-tilted algebras.

\vspace{0.3cm} \noindent \textbf{(d) Positivity.} Considering again our example, we see
that in the numerator, all monomials have positive coefficients. This has been
conjectured to be true in general. Note that even though the monomials $m_1$ and $m_2$ have positive coefficients, this property could get destroyed when putting the element $x_i'=\frac{m_1+m_2}{x_i}$ in reduced form.

\vspace{0.3cm} \noindent \textbf{(e) Clusters and seeds.} Another interesting problem is
the following, proved for finite type in \cite{fz2}.

\vspace{0.2cm} 
\noindent {\bf Problem:} Is a seed $(\ul{x}',Q')$ expressed in terms of the initial seed
$(\ul{x}, Q)$ uniquely determined by its cluster $\ul{x}'$, that is, if we know $\ul{x}'$,
then do we also know $Q'$?

\vspace{0.3cm} \noindent {\bf (f) Clusters differing only at one cluster variable.} When
applying mutation of seeds, the new cluster $\ul{x}''$ has exactly one cluster variable
which is not in the old cluster $\ul{x}'$. If we again consider the example in 1.3, we see
that if we remove a cluster variable from a cluster, there is a unique
 other cluster variable which can replace it to get  a new cluster. More generally,
the following is proved in \cite{fz2}.
\begin{theorem}
  Let $C(Q)$ be the cluster algebra associated with a Dynkin quiver $Q$. Then there
is a unique way to replace a cluster variable in a cluster by another cluster variable to
get a new cluster.
\end{theorem}

In general, there is the following problem.

\bigskip
\noindent {\bf Question:} For any cluster algebra, is there a unique way of replacing any
cluster variable in a cluster by another cluster variable to get a new cluster?

\vspace{0.3cm} We remark that in the Dynkin case it is known that the
cluster variables are in bijection with the almost positive roots, that is, the positive
roots together with the negative simple roots.


\subsection{Possible modelling}
The theory of cluster algebras has many nice features, and it is an interesting problem
to see if one can find good analogs of the main ingredients involved in their definition,
in some appropriate category $\cc$.

We want this category to be additive and to have the following properties.

(i) To have an analog of clusters we want a special class of objects, all having the same
number $n$ of nonisomorphic indecomposable summands.

(ii) To imitate the process of seed mutation, we would want that each indecomposable
summand of an object in the class can be replaced by a (unique) nonisomorphic
indecomposable object such that we still get an object in our class.

(iii) To get a categorical interpretation of the definition of the new cluster variable
$x_i''$ coming from $x_i'$, we would want that when an indecomposable object $M$ is
replaced by an indecomposable object $M^*$, there is a relationship between $M$ and $M^*$ corresponding to the formula $x_i'x_i''=m_1+m_2$. One possible way would be to have exact sequences or triangles $M^*\to B\to M$ and $M\to B' \to M^*$ with $B$ and $B'$ related to $m_1$ and $m_2$.

(iv) We would want  an interpretation of the  FZ-mutation.

\vspace{0.3cm} The hope would be that this point of view should lead to an interesting
theory in itself, and at the same time, or instead, give a better understanding of the
cluster algebras.

\vspace{0.3cm} {\bf Notes:} The material in 1.1,1.2,1.4 is taken from
\cite{fz1}\cite{fz2}\cite{fz3}\cite{bfz}; see \cite{birsc} for material related to 1.5.

\bigskip

\section{Cluster categories}
Associated with a given cluster algebra we want to find some category $\cc$ having a set
of objects which we can view as analogs of clusters and which satisfy some or all of the
requirements listed in 1.5.

A cluster algebra is said to be \textit{acyclic} if in the mutation class of the associated
cluster quivers there is some quiver $Q$ with no oriented cycles. Then we have an
associated finite dimensional hereditary $k$-algebra $kQ$. So the category $\mod kQ$ of
finite dimensional  $kQ$-modules might be a natural choice of category for modelling
acyclic cluster algebras.

\bigskip
\subsection{Tilting modules over hereditary algebras}
If we consider $\cc=\mod kQ$ as the category we are looking for, then a natural choice of
objects would be the tilting $kQ$-modules. On one hand the reason is that they have $n$
nonisomorphic indecomposable summands, where $n$ is the number of vertices in $Q$. On the
other hand there is a special tilting module associated with
 a BGP-reflection of a quiver, and as we have seen, BGP-reflection is a special case
of FZ-mutation. It will be instructive to first discuss the connection between  BGP-reflection and tilting.
Recall that for a hereditary algebra $H$, an $H$-module $T$ is \emph{tilting} if
$\Ext^1_{H}(T,T)=0$ and $T$ has exactly $n$ nonisomorphic indecomposable summands up to
isomorphism.

\bigskip
\noindent {\bf Example:} Let $Q$ be the quiver $1\to 2\to 3$ and $H=kQ$.

{\bf (a)} We first do mutation at the vertex 3. Then $\mu_3(Q)=Q_4\colon 1\to 2\leftarrow
3$. The $H$-module $H$ is clearly a tilting $H$-module. Write $H=P_1\oplus
P_2\oplus P_3$, where $P_i$ is the indecomposable projective module associated with the
vertex $i$, and let $S_i$ denote the simple top of $P_i$. Let $T= P_1\oplus
P_2\oplus\tau^{-1}S_3$, where $\tau$ denotes the translation associated with almost split
sequences, so that $\tau^{-1}S_3=S_2$. We then have the following AR-quiver.

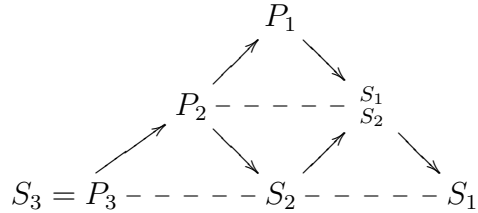
\[\xymatrix@C0.5cm@R0.5cm{
  && P_1\ar[dr]&&\\
  & P_2\ar[ur]\ar[dr]&& {\begin{smallmatrix} S_1\\S_2
\end{smallmatrix}}\ar[dr]\ar@{--}[ll]&\\
  S_3=P_3\ar[ur]&& S_2\ar[ur]\ar@{--}[ll]&& S_1\ar@{--}[ll]
}\]

Note that $\End_{H}(T)^{\op}\simeq kQ'=H'$, and $\End_{H}(H)^{\op}\simeq kQ$. So we can
pass from $kQ$ to $kQ_4$, and hence from $Q$ to $Q_4$, by replacing the indecomposable
summand $P_3$ of the tilting module $H$ by $\tau^{-1}S_3$ to get another tilting
$H$-module, and then taking endomorphism algebras. Note that $\tau^{-1}P_3$ is the only
indecomposable $H$-module which can replace $P_3$ to give a new tilting module.

This example illustrates the module theoretical interpretation of the BGP-reflection
functors. The functor $\Hom_H(T, \ )\colon \mod H\to \mod H'$ induces an equivalence $\mod_{P_3}H\to \mod_{S'}H'$, where the indecomposable modules in
$\mod_{P_3}H$ are those in $\mod H$ except $P_3$, and the indecomposable modules in
$\mod_{S'}H'$ are those in $\mod H'$, except some simple injective $H'$-module $S'$.
Hence we also get a close connection between the AR-quivers, and the AR-quiver for $H'$
is the following

$$
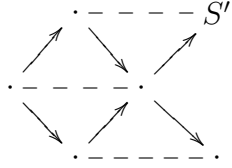
\xymatrix@C0.5cm@R0.5cm{
  & \cdot\ar[dr]&& S'\ar@{--}[ll]\\
  \cdot\ar[ur]\ar[dr]&&\cdot\ar[ur]\ar[dr]\ar@{--}[ll]&\\
  & \cdot\ar[ur]&& \cdot\ar@{--}[ll]
}$$ 

 {\bf (b)} We now do FZ-mutation at vertex 2 in $Q$, and get
$\mu_2(Q)=Q_3\colon \xymatrix@R0.2cm@C0.4cm{ 1\ar@/^/[rr]&2\ar[l]&3\ar[l] }$. Then it is
natural to try to replace  $P_2$ in $H=P_1\oplus P_2\oplus P_3$ to see if we can get a
nonisomorphic tilting module, and if there is a unique one. This is indeed the case, and
the new tilting module is $T=P_1\oplus S_1\oplus P_3$. But here we have maps $P_3\to
P_1\to S_1$ with zero composition, so that $\End_{H}(T)^{\op}$ is given by the quiver
with relations $\xymatrix@R0.2cm@C0.4cm{ 1\ar@/^0.8pc/[rr]&2\ar@{}@<-1.1ex>[r]^(0.5){\cdots\cdots}\ar[l]&3 }$, where
an arrow $3\to2$ is missing compared to $Q_3$.

So our procedure does not work from the point of view of getting a model for the
FZ-mutation, but it is quite close to working. What we would need is to have more maps in
our category than  what we have in $\mod H$, in particular we would like to have nonzero
maps from $S_1$ to $P_3=S_3$.

{\bf (c)} We also consider $\mu_1(Q)=Q_2\colon 1\leftarrow 2\to 3$ from the same point of view.
Now we would like to replace $P_1$ in $H=P_1\oplus P_2\oplus P_3$ with another
indecomposable $H$-module to  obtain a tilting module. But here we encounter a problem at
an earlier stage. This is actually not possible. The general explanation is that a
projective injective module has to be a direct summand of any tilting module. So here we
get a problem which indicates that the category $\mod H$ is not large enough for being
able to replace $P_1$.

In conclusion, as illustrated by this example, there are the following problems with
using the tilting modules over hereditary algebras as a model for clusters.

(1) There are not enough objects in order to replace any indecomposable summand of a
tilting module with a nonisomorphic indecomposable module to get a new tilting module.

(2) The quiver of the endomorphism algebra of a tilting module is not the desired one,
the problem being that there are not enough maps.

\bigskip

We call an $H$-module $\ol{T}$ with $\Ext^1_H(\ol{T},\ol{T})=0$ and with $n-1$ nonisomorphic indecomposable summands  an \emph{almost complete tilting} $H$-module. An indecomposable $H$-module $M$ such that $\ol T\oplus M$ is a tilting module is called a \textit{complement} of $\ol T$.  The following is known for tilting $H$-modules.
\begin{theorem}
 (a) If $T$  is a tilting $H$-module, then each indecomposable summand $M$ can be
replaced by at most one nonisomorphic indecomposable $H$-module to get a new tilting module.

(b) There are exactly two complements for $T/M$ if and only if $T/M$ is sincere, that is, each simple
$H$-module occurs as a composition factor.
\end{theorem}

Note that in our example $P_1\oplus P_2$ and $P_1\oplus P_3$ are sincere, whereas
$P_2\oplus P_3$ is not. 

In the case when an almost complete tilting module $\ol{T}$ has two complements, that is,
there are two ways of completing it to a tilting module, they are connected as follows:

\begin{theorem}
  Let $\overline{T}$ be an almost complete tilting $H$-module and $M$ and $M^*$
nonisomorphic indecomposable modules such that $\ol{T}\oplus M$ and $\ol{T}\oplus M^*$
are  tilting modules. Then there is an exact sequence $0\to M^*\xrightarrow{g} B \xrightarrow{f} M \to 0$ where
$f\colon B\to M$ is a minimal right $\add\ol{T}$-approximation and $g\colon M^*\to B$ is
a minimal left $\add\ol{T}$-approximation or an exact sequence $0\to M\to B'\to M^*\to 0$
with the corresponding properties.
\end{theorem}

There is an important class of algebras associated with tilting  modules over hereditary
algebras. An algebra is said to be \emph{tilted} if it is of the form
$\End_{H}(T)^{\op}$, where $T$ is a tilting module over a finite dimensional hereditary
algebra $H$. These algebras appear frequently in
representation theory, and they are close enough to hereditary algebras to inherit nice
properties.

For an $H$-module $T$, denote by $\Fac T$ the subcategory  of $\mod H$ whose objects are
factors of finite direct sums of copies of $T$ and by $\Sub T$ the subcategory whose objets are submodules of finite direct sums of copies of $T$. Recall also that a subcategory of $\mod
H$ is a \textit{torsion class} if it is closed under factors and extensions, and a \textit{torsionfree}
class if it is closed under submodules and extensions. Then we have the following
relationship between hereditary algebras and tilted algebras.

\begin{theorem}
  Let $H$ be a hereditary finite dimensional algebra, and $T$ be a tilting $H$-module,
and $\la=\End_{H}(T)^{\op}$.

\noindent (a) $\ct=\Fac T$ is a torsion class in $\mod H$, with associated torsionfree
class $\cF=\{X;\Hom_{H}(T,X)=0\}$, so that $(\ct,\cF)$ is a \textit{torsion pair}.

\noindent (b)There exists a \textit{torsion pair} $(\cx,\cy)$ in $\mod\la$, where $\cy=\Sub D(T)$ such that
\begin{itemize}
  \item[(i)]$\Hom_H (T,\text{ })\colon\mod H\to \mod\la$ induces an equivalence
between $\ct$ and $\cy$
\item[(ii)]$\Ext^1_{H}(T, \text{ })\colon\mod H\to \mod\la$ induces an equivalence
between $\cF$ and $\cx$
\item[(iii)] each indecomposable object in $\mod\la$ is in $\cx$ or $\cy$.
\end{itemize}
\end{theorem}

An important homological property which can be proved for a tilted algebra is that it has global dimension at most 2.

\bigskip
\subsection{Definition and examples}
The question is now how to modify the category $\mod H$ to take care of the shortcomings
discussed in  Section 2.1. In addition we know from \secI that for cluster algebras given
by Dynkin quivers, the cluster variables are in one-one correspondence with the almost
positive roots. Hence there are $n$ more cluster variables than the number of
indecomposable modules for the Dynkin quiver, where $n$ is the number of vertices in the quiver. We now explain how to modify $\mod H$ in view of of the above remarks.

Let ${\bf D}^b(H)$ be the bounded derived category of the finite dimensional hereditary 
$k$-algebra $H=kQ$, where $Q$ is a finite connected quiver without oriented cycles. Then the
indecomposable objects are all isomorphic to stalk complexes. The translation $\tau$,
which in this case gives an equivalence from the category $\mod_{\cp}H$ whose
indecomposable $H$-modules are not projective to the category $\mod_{\ci}H$ whose
indecomposable $H$-modules are not injective, induces an equivalence $\tau\colon {\bf
D}^b(H)\to {\bf D}^b(H)$. Then $\tau(C)$ is the left hand term of the almost split
triangle with right hand term $C$. Note that under the embedding $\mod H\to {\bf D}(H)$,
almost split sequences go to almost split triangles (see \cite{ha1}).

Let now $F$ be the equivalence $\tau^{-1}[1] $ from ${\bf D}^b (H )$ to ${\bf D}^b(H)$,
where [1] is the shift functor. Then we define the cluster category $\cc_{H}$ to be the
orbit category ${\bf D}^b(H)/F$. The objects in $\cc_{H}$ are the same as those in ${\bf D}^b(H)$. If $A$ and $B$ are in ${\bf D}^b(H)$, then by definition we have $\Hom_{\cc_H}(A,B)=\oplus_{i\in\Z}\Hom_{{\bf D}^b(H)}(A,F^iB)$. Consider the \emph{fundamental
domain} of indecomposable objects given by $\ind H \vee \{P_i[1]; i=1,\cdots, n\}$, where
$P_1,\cdots, P_n$ are the nonisomorphic indecomposable projective $H$-modules. It is easy
to see that each indecomposable object in $\cc_H$ is isomorphic to exactly one of these
indecomposable objects. When $A$ and $B$ are chosen from this fundamental
domain, the formula for $\Hom_{\cc_{H}}(A,B)$ simplifies to $\Hom_{{\bf D}^b(H)}(A,B)\oplus\Hom_{{\bf D}^b(H)}(A,FB)$. We illustrate with the following.

\bigskip
\noindent {\bf Example:} Let $Q$ be the quiver $1\to 2\to 3$, and let $S_i$ and $P_i$ be
the simple and indecomposable projective $H$-modules corresponding to the vertex $i$,
where $H=kQ$. We then have the following AR-quiver for $H$, and for ${\bf D}^b(H)$

$$\xymatrix@C0.5cm@R0.5cm{
  &&& \bullet\ar@{-->}[dr]&& P_1\ar[dr]&& \bullet
\ar[dr]&&\makebox[0mm]{$S_2[1](=\tau^{-1}S_3[1])$} \ar[dr]&& \bullet &\\
  \ar@{..}[rr]&&\bullet \ar@{-->}[dr]\ar[ur]&& P_2 \ar[ur]\ar[dr] && {\begin{smallmatrix}
S_1\\S_2 \end{smallmatrix}}\ar@{-->}[ur] \ar[dr]&& \bullet\ar[ur]\ar[dr]&&
\bullet\ar[ur]\ar@{..}[rr]&&\\
  & \bullet \ar[ur]&& \makebox[0mm]{$P_3=S_3$} \ar[ur]&& S_2 \ar[ur]&& S_1\ar@{-->}[ur]&& \bullet\ar[ur]&&&
}$$

Then we have $$\Hom_{\cc_{H}}(S_1,S_3)=\Hom_{{\bf D}^b(H)}(S_1, S_3)\oplus\Hom_{{\bf
D}^b(H)}(S_1,\tau^{-1}S_3[1])= \Ext^1_{H}(S_1,S_2)\simeq k.$$

Note that considering again Example (b) in Section 2.1, we see that $\End_{\cc_{H}}(P_1\oplus
S_1\oplus S_3)^{\op}$ has indeed the quiver $\xymatrix@C0.5cm{
 \bullet \ar[r] &\bullet\ar[r]&\bullet\ar@/^/[ll]
}$, due to the extra maps from $S_1$ to $S_3$.

Also the problem about complements in Example (c) in Section 2.1 can now be solved, with an
appropriate notion of ``tilting'' objects. We have that $T=P_1[1]\oplus P_2\oplus P_3$ is
an object in $\cc_{H}$ with $\Ext^1_{\cc_{H}}(T,T)=0$.

\bigskip
We next give an example to show that for $\Hom_{\cc_H}(A,B)=\Hom_{{\bf
D}^b(H)}(A,B)\oplus\Hom_{{\bf D}^b(H)}(A,FB)$, where $A$ and $B$ are in the fundamental
domain, there can be nonzero maps in both summands.

\bigskip

\noindent {\bf Example:} Let $Q$ be the quiver
$$
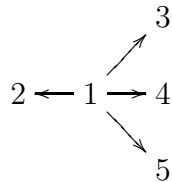
\xymatrix@C0.5cm@R0.5cm{
  &&3\\
2& 1\ar[l]\ar[ur]\ar[dr]\ar[r]& 4\\
&& 5 }$$ and $H= kQ$ the associated path algebra. Let $M$ be the indecomposable module
$\left(\begin{smallmatrix}
&&\\
&1&\\
2&&3\\
&&
\end{smallmatrix}\right)
$.  Then $M$ lies in a tube of rank two, and we have $\tau \left(\begin{smallmatrix}
&&\\
&1&\\
2&&3\\
&&
\end{smallmatrix}\right)
= \left(\begin{smallmatrix}
&&\\
&1&\\
4&&5\\
&&
\end{smallmatrix}\right)$. For computing $\Hom_{\cc_H}(M,M)$, we have
$\Hom_H(M,M)\neq 0$ and $\Hom_{{\bf D}^b(H)}(M,\tau^{-1}M[1])=\Ext^1_{H}(M,
\tau^{-1}M)\simeq D\Hom_H(\tau^{-2}M,M)\simeq D\Hom_H(M,M)\neq 0$.

\bigskip
The following properties of cluster categories will turn out to be important.

\begin{theorem}
\begin{enumerate}
\item[(a)] The cluster categories are triangulated categories, and the natural functor
from $\DBH$ to $\CC$ is triangulated.
\item[(b)] The cluster category $\cc_{H}$ has almost split triangles, and they are
induced by almost split triangles in ${\bf D}^b(H)$.
\item[(c)] For $A$, $B$ in $\cc_H$ we have a functorial isomorphism $D\Ext^1_{\cc_H}(A,B)\simeq \Ext^1_{\cc_H}(B,A)$.
\item[(d)] For $A$, $B$ in $\mod H$ we have $\Ext^1_{\cc_H}(A,B)\simeq\Ext^1_H(A,B)\oplus D\Ext^1_H(B,A)$.
\end{enumerate}
\end{theorem}

Note that part (a) is highly nontrivial, and it is not true in general that  orbit
categories of ${\bf D}^b(\la)$ for any finite dimensional algebra $\la$ are triangulated.

While any almost split triangle in $\cc_H$ is induced by an almost split triangle in
${\bf D}^b(H)$, it is not true that any triangle in $\cc_H$ comes from a triangle in
${\bf D}^b(H)$. This is for example not the case for the triangle induced by a map
$(f,g)\colon M\to M$ in the previous example, where $f$ and $g$ are nonzero. This in one
reason why it is difficult to show that $\cc_H$ is triangulated.

There are orbit categories of ${\bf D}^b(H)$ which were previously known to be
triangulated, namely the stable categories $\ul{\mod}\la$ for selfinjective algebras of
finite type in \cite{rie}, which are triangulated by \cite{ha1}. The cluster categories are not of this form, but this still
gave an indication that the same thing might be true for cluster categories. In addition
the orbit categories  ${\bf D}^b(H)/[2]$ were
known to be triangulated \cite{px}.
\bigskip
\subsection{Cluster-tilting objects}
We need to define the objects in $\cc_{H}$ which should replace tilting $H$-modules. It
would be desirable if the tilting modules when viewed in $\cc_{H}$ would belong to this
class.

It turns out to be natural to consider the condition of \textit{maximal rigid}, that is, $\Ext^1_{\cc_{H}}(T,T)=0$ and $T$ is maximal with this property. The relationship to tilting
modules is given by the following. Note that derived equivalent hereditary algebras have
equivalent cluster categories.

\begin{theorem}\label{cluster-tilting objects}
  The maximal rigid objects in the cluster category $\cc_{H}$ are exactly those
coming from tilting modules over some hereditary algebra $H'$ derived equivalent to $H$.
\end{theorem}

An object $T$ in $\cc_H$ is called \textit{cluster-tilting} if $\Ext^1_{\cc_H}(T,T)=0$ and $\Ext^1_{\cc_H}(T,M)=0$ implies that $M$ is in $\add T$. Clearly any cluster-tilting object is maximal rigid. But for cluster categories these concepts actually coincide.

\begin{proposition}
  An object $T$ in the cluster category $\CC$ is cluster-tilting  if and only if it
is maximal rigid.
\end{proposition}
\begin{proof}
Assume that $T$ is maximal rigid. By Theorem \ref{cluster-tilting objects} we can assume that
$T$ is a tilting $H$-module. Assume that $\Ext^1_{\CC}(T,M)=0$ when $M$ is indecomposable. If $\Ext^1_{\CC}(M,M)=0$, then $M$ is in $\add T$. If $\Ext^1_{\CC}(M,M)\neq 0$, we can assume that $M$ is an $H$-module since for all $P[1]$ with $P$ indecomposable projective we have $\Ext^1_{\CC}(P[1],P[1])=0$. Since then 
$\Ext_{H}^1(T,M)=0=\Ext_{H}^1(M,T)$ by Theorem 2.4, tilting theory gives that $M$ is in $\Fac T\cap
\Sub T=\add T$.
\end{proof}

In \cite{bmrrt} the term cluster-tilting object was used for the above concept of maximal rigid. What is here called cluster-tilting object corresponds to a finite set of indecomposable objects being an $\Ext$-configuration in the sense of \cite{bmrrt}. This concept was motivated by the $\Hom$-configurations of Riedtmann. The above definition of cluster-tilting object was used in a more general context in \cite{kr1}. It is closely related to Iyama's definition of maximal
1-orthogonal modules, which he introduced in connection with his
generalizations of the theory of almost split sequences, originally in the abelian case,
where the notion was essential \cite{i1}\cite{i2}.

We say that $\ol{T}$ is an \emph{almost complete cluster-tilting object} if there
is some indecomposable object $M$ not in $\add \ol T$ such that $\ol{T}\oplus M$ is a
cluster-tilting object. Then $M$ is said to be a complement of $\ol{T}$. It follows from Theorem \ref{cluster-tilting objects} that all cluster-tilting objects in
a given $\CC$ for $H=kQ$ have the same number of nonisomorphic indecomposable summands as
the vertices in the quiver $Q$. Also we now get a better result on exchanging
indecomposable summands of cluster-tilting objects, more closely related to clusters.

\begin{theorem}
  Let $\ol{T}$ be an almost complete cluster-tilting object in $\CC$. Then $\ol{T}$
has exactly two nonisomorphic complements in $\CC$.
\end{theorem}

\begin{proof}
We only show that there are at least two complements. So let $\ol T$ be an almost complete tilting $H$-module. If $\ol T$ is sincere, it has exactly two complements in $\mod H$, and by Theorem 2.4(b) these are also complements in $\CC$. If $\ol T$ is not sincere, there is exactly one complement in $\mod H$, and this is a complement also in $\CC$. There is some simple $H$-module $S$ which is
not a composition factor of $\ol{T}$, so that $\Hom_H(P, \ol{T})=0$ where $P$ is the
projective cover of $S$. Hence we have
$\Ext^1_{\CC}(P[1],\ol{T})=\Hom_{\CC}(P,\ol{T})=0$, and so $P[1]$ is a complement.

\end{proof}

There is a graph, called the \textit{cluster-tilting graph}, 
where the vertices are the (nonisomorphic) cluster-tilting objects and there is an edge
between two vertices if the corresponding cluster-tilting objects have a common almost
complete cluster-tilting summand. Then we have the following important result.

\begin{theorem}
For a cluster category $\CC$ the cluster-tilting graph is connected.
\end{theorem}

\bigskip
\subsection{Exchange pairs}
Let $\ol{T}$ be an almost complete cluster-tilting object in a cluster category $\CC$,
and let $M$ and $M^*$ be the nonisomorphic complements for $\ol{T}$. We shall now
investigate the relationship between  $M$ and $M^*$.

We have the following connection.

\begin{theorem}
  Let the notation be as above. Then there exist triangles $M^*\xto{f}B\xto{g}M\to$
and $M\xto{s}B'\xto{t}M^*\to $, where $g\colon B\to M$ and $t\colon B'\to M^*$ are
minimal right $\add \ol{T}$-approximations and $f\colon M^*\to B$ and $s\colon M\to B'$
are minimal left $\add \ol{T}$-approximations.
\end{theorem}

We illustrate with the following.

\bigskip
\noindent {\bf Example:} Consider again $\CC$ for $H=kQ$, where $Q\colon 1\to 2\to 3$.
Let $\ol{T}=P_3\oplus P_1$. Then the two complements are $M=P_2$ and $M^*=S_1$. The
triangles connecting $M$ and $M^*$ are of the form 
$$S_1\to S_3\to P_2\to \text{ }\text{  and }\text{ }P_2\to P_1\to S_1\to $$
where $\Hom_{\CC}(S_1,S_3)=\Hom(S_1,\tau^{-1}S_3[1])=\Hom(S_1,S_2[1])\simeq k$.

\bigskip
When $M$ and $M^*$ are complements of a common  almost complete cluster-tilting object,
we call $(M,M^*)$\textit{ an exchange pair}. There is the following characterization of
such a pair.

\begin{theorem}
  A pair of indecomposable objects $(M,M^*)$ in a cluster category $\CC$ is an
exchange pair if and only if $\Ext^1_{\CC}(M,M^*)$ is one dimensional over $k$.
\end{theorem}

 Note that when we do not work over an algebraically closed field then the relevant condition is
that $\Ext^1_{\CC}(M,M^*)$ is one dimensional over $\End_{\CC}(M)/\rad\End_{\CC}(M)$ and over
$\End_{\CC}(M^*)/\rad\End_{\CC}(M^*)$.
\bigskip
\subsection{Analogs}
We have seen that in the cluster category $\CC$ we have described a collection of
objects, which all have the same number of nonisomorphic indecomposable summands, the
same way as all clusters have the same number of elements. In both cases these numbers
coincide with the number of vertices in the quiver.

We note a slight difference with respect to exchange. For cluster-tilting  objects there
is a unique way of exchanging an indecomposable summand.  For clusters there is by
definition at least one way of exchanging a cluster variable to get a new cluster, but it
is not clear that it is unique. There might be related clusters at other places in the cluster graph.

The analog of cluster variables is now clearly the indecomposable summands of the
cluster-tilting objects, which are the indecomposable rigid objects.

As analogs of the seeds $(\ul{x}', Q')$ we have \emph{tilting seeds} $(T,Q_{T})$, where
$T$ is a cluster-tilting object and $Q_T$ is the quiver of $\End_{\CC}(T)^{\op}$. Note
that also here there is a slight difference, since a tilting seed by definition is 
determined by the cluster-tilting object, while the corresponding result is not known in
general in the context of cluster algebras, as highlighted in Section 1.4(e). Actually, these
differences in behavior also indicate some strength, and can be used to prove new results
on cluster algebras.

The exchange triangles $M^*\to B\to M\to $ and $M\to B'\to M^*\to $ are connected with
the exchange multiplication $x_ix_i^*=m_1+m_2$ in cluster algebras. Here the monomials
$m_1$ and $m_2$ correspond to the objects $B$ and $B'$.

\smallskip
We point out that for almost all results in this chapter we could instead deal with
hereditary abelian categories with  finite dimensional homomorphism and extension spaces
and which have a tilting object. By \cite{ha2} it is known that the only additional
categories we have to deal with are the categories coh $\mathbb{X}$ of coherent sheaves
on weighted projective lines \cite{gl}.

The only result which remains open in this setting is whether the cluster-tilting graph
is connected.

\bigskip
\noindent {\bf Notes:}  The results from tilting theory are taken from
\cite{apr}\cite{bb}\cite{hr}\cite{hu1}\cite{hu2}\cite{rs}\cite{u}. It was proved in
\cite{k} that the cluster categories are triangulated. Otherwise the material in this
section is taken from \cite{bmrrt}\cite{bmr2}.

\bigskip
\section{Cluster-tilted algebras}

In the same way as the class of tilted algebras is defined as endomorphism algebras of
tilting modules over hereditary algebras, we consider endomorphism algebras of
cluster-tilting objects in cluster categories. These algebras have been called
\emph{cluster-tilted} algebras. They have several interesting properties, ranging from
homological properties to properties described in terms of quivers with relations. In particular
there are nice relationships with the associated hereditary algebras.

\bigskip
\subsection{The quivers of the cluster-tilted algebras}
We first note that the hereditary algebra $H$ is itself a cluster-tilted algebra since
$\Hom_{\CC}(H,H)=\Hom_H(H,H)\oplus\Hom_{{\bf D}^b(H)}(H,\tau^{-1}H[1])$, where the last
term is clearly zero. An important property of the quiver $Q_T$ of a cluster-tilted algebra $\End_{\CC}(T)^{\op}$ is that $Q_T$ has no loops or $2$-cycles. Otherwise the basis for information on the quivers of cluster-tilted algebras
comes from comparing the cluster-tilted algebras $\Gamma = \End_\CH(T)^{\op}$ and
$\Gamma' = \End_\CH(T')^{\op}$, where $T$ and $T'$ are nonisomorphic cluster-tilting
objects having a common almost complete cluster-tilting object, that is, $T$ and $T'$ are
neighbours in the cluster-tilting graph.

\begin{theorem} \label{3.1}
With the above notation, let $Q_T$ be the quiver of $\Gamma = \End_\CH(T)^{\op}$ and
$Q_{T'}$ the quiver of $\Gamma' = \End_\CH(T')^{\op}$. Write $T = T_1 \oplus \dots \oplus
T_n$, where $T_i$ are nonisomorphic indecomposable objects, and $T_k$ is not a summand of
$T'$. Then $\mu_k(Q_T) = Q_{T'}$.
\end{theorem}

Using this, we get the following consequence, where we use that the cluster-tilting graph
is connected.

\begin{theorem} \label{3.2}
Let $Q$ be a finite connected quiver without oriented cycles, and let $\CH$ be the cluster category
associated with $H = kQ$. Then the quivers of the cluster-tilted algebras associated with
$\CH$ are the quivers in the mutation class of $Q$.
\end{theorem}

\noindent {\bf Example.} The quivers of the cluster-tilted algebras of type $A_3$ are
$$
\begin{array}{c}{\xymatrix@C0.6cm{\bullet \ar[r] & \bullet \ar[r] & \bullet}}\end{array},
\begin{array}{c}{\xymatrix@C0.6cm{\bullet \ar[r] & \bullet & \bullet \ar[l]}}\end{array},
\begin{array}{c}{\xymatrix@C0.6cm{\bullet & \bullet \ar[l] \ar[r] &
\bullet}}\end{array}\textrm{ and}
\begin{array}{c}{ \xymatrix@C0.5cm{
 \bullet \ar@/^/[rr] &\bullet\ar[l]&\bullet\ar[l]
}}\end{array}
$$
as already discussed in Section 1.3.

\bigskip
The following result on cluster-tilted algebras, also of interest in itself, is useful
for proving Theorem~\ref{3.1}.

\begin{proposition} \label{3.3}
If $\Gamma$ is a cluster-tilted algebra, then for any sum $e$ of vertices in the quiver,
we have that $\Gamma/\Gamma e \Gamma$ is also cluster-tilted.
\end{proposition}

We note that here there is a difference as compared to tilted algebras, where the
corresponding result is not true in general. On the other hand the class of tilted
algebras is closed under taking endomorphism algebras of projective modules, while this
is not in general the case for cluster-tilted algebras.

The above result makes it possible to reduce the proof of Theorem~\ref{3.1} to the case
of cluster-tilted algebras with $3$ simple modules. There is a description
of the possible quivers of cluster-tilted algebras with 3 vertices in terms of Markov equations \cite{bbh}, and there is more
information on these algebras in \cite{ker2}.

Theorem~\ref{3.2} establishes a nice connection between cluster-tilting theory and
cluster algebras. We shall see some applications in the next section.

\bigskip
\subsection{Relations}

It is of course also of interest to describe the relations for a cluster-tilted algebra
once the quiver is given. The following has recently been proved \cite{birsm}.

\begin{theorem}
A cluster-tilted algebra is uniquely determined by its quiver.
\end{theorem}

The conjecture was first verified in the case of finite representation type. In this case
an explicit description is given of a set of minimal relations. Note that the quiver $Q$
has only single arrows. For each arrow $\alpha: i \to j$ in the quiver which lies on a
full oriented cycle, that is, a cycle where there are no other arrows in $Q$ between the
vertices of the cycle, take the sum of the paths from $j$ to $i$, which together with
$\alpha$ give full cycles. For a given $\alpha$ it turns out to be at most two such cycles. Then these relations determine the corresponding
cluster-tilted algebra.

We give some examples to illustrate.

\bigskip
\noindent {\bf Example} Let $Q$ be the quiver
$$
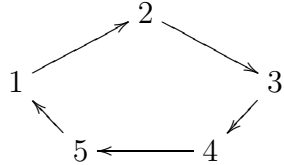
\xymatrix@C0.4cm@R0.4cm{
            &           & 2 \ar[drr]    &           &               \\
1 \ar[urr]  &           &               &           & 3 \ar[dl]     \\
            & 5 \ar[ul] &               & 4 \ar[ll] &
}
$$
which is the quiver of a cluster-tilted algebra $\Gamma$ of type $\DV$. This can be seen
by finding a sequence of mutations from $Q$ to a quiver of type $\DV$. Then $\Gamma$ is
of finite representation type, and is determined by the relations given by all
paths of length 4 being zero.

\bigskip
\noindent {\bf Example} Let $Q$ be the quiver
$$
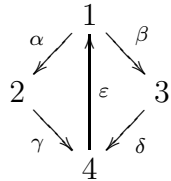
\xymatrix@C0.5cm@R0.5cm{
                    & 1 \ar[dl]_\alpha \ar[dr]^\beta        &                   \\
2 \ar[dr]_\gamma    &                                       & 3 \ar[dl]^\delta  \\
                    & 4 \ar[uu]_\varepsilon                                 &
}
$$
Then $Q$ is mutation equivalent to a quiver of type $\DIV$, and the corresponding
cluster-tilted algebra is hence of finite representation type. Then the relations are
$\gamma\alpha + \delta\beta = 0$, $\varepsilon\gamma = 0$, $\varepsilon\delta = 0$,
$\alpha\varepsilon = 0$, $\beta\varepsilon = 0$.

\bigskip

Some of these results about the relations hold more generally, as given in the following.

\begin{proposition} \label{3.4}
Let $\Gamma$ be a cluster-tilted algebra and $S_1$, $S_2$ simple $\Gamma$-modules. Then
we have $\dim \Ext^1_\Gamma(S_1,S_2) \geq \dim \Ext^2_\Gamma(S_2,S_1)$.
\end{proposition}

It is known that if there is a relation from  the vertex of $S_2$ to the vertex of $S_1$ in a minimal set of relations for $\Gamma$, then $\Ext^2_{\Gamma}(S_2,S_1)\neq 0$ (see \cite{bo}\cite{birsm}). Hence there must be an arrow from the vertex of $S_1$ to the vertex of $S_2$ in the quiver of $\Gamma$. But conversely, there may not be such a relation for each arrow lying on a cycle.

\bigskip

\noindent {\bf Example} The quiver
$$
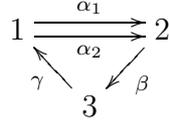
\xymatrix@C0.5cm@R0.5cm{ 1 \ar@<.5ex>[rr]^{\alpha_1} \ar@<-.5ex>[rr]_{\alpha_2}   & & 2
\ar[dl]^\beta   \\
                                                        & 3 \ar[ul]^\gamma  &
}
$$
is mutation equivalent to
$$
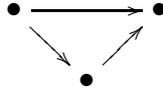
\xymatrix@C0.5cm@R0.5cm{
\bullet \ar[rr] \ar[dr]   &                   & \bullet     \\
                        & \bullet \ar[ur]     &
}
$$
and is hence the quiver of a cluster-tilted algebra. A minimal set of relations giving
rise to a cluster-tilted algebra is given by $\alpha_2\gamma = 0 = \beta\alpha_2 =
\gamma\beta$. Here we have $\dim \Ext^2_\Gamma(S_2,S_1) = 1 < 2 = \dim \Ext^1_\Gamma(S_1,S_2)$, and there is no relation associated with the arrow $\alpha_1$.

\bigskip

\subsection{Relationship with hereditary algebras}

There is a close relationship between cluster-tilted algebras and the associated
hereditary algebras, as we shall now discuss.

For tilted algebras $\Lambda$ there is a close relationship with the corresponding
hereditary algebras $H$, where two subcategories of $\modL$, coming from a torsion pair,
are also equivalent to subcategories of $\mod H$ belonging to a torsion pair. Here we
may have that $H$ is of infinite type, while $\Lambda$ is of finite type. Roughly
speaking, in the case of cluster-tilted algebras we have however the ``same number'' of
indecomposables. Here we first enlarge the category $\modH$ by passing to $\CH$ and hence
adding $n$ indecomposable objects, where $n$ is the number of nonisomorphic simple
$H$-modules. Then we ``remove'' $n$ other indecomposable objects from $\CH$ to obtain
$\modG$.

When $T$ is a tilting $H$-module and $\Lambda = \End_H(T)^\op$, then, as we have pointed
out, the functor $\Hom_H(T,\;): \modH \to \modL$ induces an equivalence between the
torsion class $\Fac T$ in $\modH$ and a torsionfree class in $\modL$, actually $\Sub
D(T)$. So this functor is far from being dense in general. But the situation is quite
different if we replace $\modH$ by the cluster category $\CH$, and the tilted algebra
$\Lambda$ by the cluster-tilted algebra $\Gamma = \End_\CH(T)^\op$. Here the triangulated
structure is important. While for an $H$-module $C$ there is usually no exact sequence
$T_1 \to T_0 \to C \to 0$ in $\modH$ with $T_0$ and $T_1$ in $\add T$, we have the
following crucial property for $\CH$. Here $\add T$ has as objects the summands of finite direct sums of copies of $T$.

\begin{lemma} \label{3.5}
Let $H = kQ$ be a finite dimensional hereditary $k$-algebra, and $\CH$ the associated
cluster category. Then for any $C$ in $\CH$ there is a triangle $T_1 \to T_0 \to C \to
T_1[1]$ with $T_0$, $T_1$ in $\add T$.
\end{lemma}

\begin{proof}
Let $f: T_0 \to C$ be a right $\add T$-approximation in $\CH$, and complete to a triangle
$X \to T_0 \to C \to X[1]$. Apply $G = \Hom(T,\;)$ to get the exact sequence $(T,X) \to
(T,T_0) \xrightarrow{(T,f)} (T,C) \to (T,X[1]) \to (T,T_0[1]) = 0$. Since $(T,f)$ is
surjective, it follows that $\Ext^1(T,X) = \Hom(T,X[1]) = 0$, so that $X \in \add T$.
\end{proof}

Using this lemma one can show the following close relationship between $\CH$ and the
cluster-tilted algebra $\Gamma = \End_\CH(T)^\op$.

\begin{theorem} \label{3.6}
Let $H$ be a finite dimensional hereditary $k$-algebra, $T$ a cluster-tilting object in
the cluster category $\CH$, and $\Gamma = \End_\CH(T)^\op$ the associated cluster-tilted
algebra. Then $G = \Hom_\CH(T,\;): \CH \to \mod \Gamma$ induces an equivalence of
categories $\overline G: \Hom_\CH(T,\;): \CH/\add \tau T \to \mod \Gamma$.
\end{theorem}

\begin{proof}
We illustrate the use of Lemma~\ref{3.5} by giving a proof of this theorem. So let
$C$ be in $\modG$, and consider a (minimal) projective presentation $(T,T_1)
\xrightarrow{(T,f)} (T,T_0) \to C \to 0$, where $f: T_1 \to T_0$ is a map in $\add T$.
Complete to a triangle $T_1 \to T_0 \to X \to T_1[1]$, and apply $\Hom_\CH(T,\;)$ to get
the exact sequence $(T,T_1) \to (T,T_0) \to (T,X) \to (T,T_1[1]) = 0$, so that $C \simeq
(T,X)$. This shows that $G$ is dense.

Then we show that the induced functor $\ol G \colon \CC/ \add \tau T\to \mod \Gamma$ is full and faithful. So let $X$ and $Y$ be objects in $\CC$. We have as before a triangle $T_1\to T_0\to X\to T_1[1]$, which induces an exact sequence $(T,T_1)\to (T,T_0)\to (T,X)\to 0$, and hence the following exact commutative diagram of $\Gamma$-modules

\[\xymatrix@C=12pt{0\ar[r]&(T_1[1],Y)\ar[r]&(X,Y)\ar[r]\ar[d]&(T_0,Y)\ar[r]\ar[d]^-{u} & (T_1,Y)\ar[d]^-{v}\\
&& (G(X),G(Y))\ar[r]&((T,T_0),G(Y))\ar[r]&((T,T_1),G(Y)).
}\]
Here $u$ and $v$ are isomorphisms since we have a natural isomorphism $(T,Y)\simeq ((T,T),(T,Y))=((T,T),G(Y))$. Hence we get an exact commutative diagram 
\[\xymatrix@C=12pt{(T_1[1],Y)\ar[r]\ar@{=}[d]&(X,Y)\ar[r]\ar@{=}[d]&(G(X),G(Y))\ar[r]&0\\
(T_1[1],Y)\ar[r]&(X,Y)\ar[r]&\Hom_{\CC/\add \tau T}(X,Y)\ar[u]\ar[r]&0
}\]
where the first sequence comes from the previous diagram and the second one from the definition of morphisms in $\CC/\add \tau T$. Hence there is induced an isomorphism $\Hom_{\CC/\add \tau T}(X,Y)\to (G(X),G(Y))$ showing that $\ol G$ is full and faithful. 
\end{proof}

This result has some similarity with the equivalences associated with the BGP reflection
functors, which induce equivalences between subcategories obtained by leaving out only
one indecomposable object, and where the AR-quivers are closely related. And actually in
this context of cluster-tilted algebras  there is also a surprisingly close connection
between the AR-quivers for $H$ and for $\Gamma$, via $\CH$. This can be used to rule out
the possibility for a given algebra to be cluster-tilted.

\begin{theorem} \label{3.7}
Let the notation be as before. Then the AR-quiver for $\Gamma$ is obtained by dropping,
in the AR-quiver of $\CC$, the vertices corresponding to the objects $\tau T_i$ for the indecomposable
summands $T_i$ of the cluster-tilting object $T$.
\end{theorem}
 
We then get the following information about cluster-tilted algebras. 
\begin{corollary}
Let $T$ be a cluster-tilting object in the cluster category $\CC$, and $\Gamma=\End_{\CC}(T)^{\op}$ the corresponding cluster-tilted algebra. Let $T=T_1\oplus \cdots\oplus T_n$, where the $T_i$ are indecomposable.
 
\begin{enumerate}
\item[(a)] The indecomposable projective $\Gamma$-modules are of the form $P_i=\Hom_{\CC}(T,T_i)$, and the indecomposable injective $\Gamma$-modules of the form $I_i=\Hom_{\CC}(T,\tau^2 T_i)$. They are related by $P_i/\rad P_i\simeq \soc I_i$.
\item[(b)]  $\Gamma$ is selfinjective if and only if $\tau^2T\simeq T$, and $\Gamma$ is weakly symmetric, that is $P_i/\rad P_i\simeq \soc I_i$ if and only if $\tau^2T_i\simeq T_i$ for $i=1,\ldots,n$.
 
\end{enumerate}
\end{corollary}
\begin{proof}
$(a)$ The first claim is clear. For $P_i=\Hom_{\CC}(T,T_i)$ we have $D\Hom_{\Gamma}(P_i,\Gamma)\simeq D\Hom_{\CC}(T_i,T)\simeq\Hom_{\CC}(T,T_i[2])\simeq\Hom_{\CC}(T,\tau^2 T_i)=I_i$. This shows the other claims. 

$(b)$ This is direct consequence of $(a)$.
\end{proof}

The selfinjective cluster-tilted algebras have been classified in \cite{rin3}.

There is the following nice consequence of Theorem 3.7. Let $T=T_1\oplus\cdots\oplus T_n$ be a basic cluster-tilting object in a cluster category $\C_Q$, with the $T_j$ indecomposable. Let $T^*=T/T_i\oplus T_i^*$ be a cluster-tilting object with $T_i^*\not\simeq T_i$. Let $\Gamma=\End_{\C_Q}(T)^{\op}$ and $\Gamma^*=\End_{\C_Q}(T^*)^{\op}$. Let $S$ be the simple $\Gamma$-module associated with $T_i$ and $S^*$ the simple $\Gamma^*$-module associated with $T_i^*$.

\begin{theorem}
With the above notation and assumptions we have an equivalence of categories 
\[\mod\Gamma/\add S\to\mod\Gamma^*/\add S^*,\]
where the maps in $\mod\Gamma/\add S$ are the maps in $\mod\Gamma$ modulo the maps which factor through an object in $\add S$, and the maps in $\mod \Gamma^*/\add S^*$ are the maps in $\mod \Gamma^*$ modulo the maps which factor through an object in $\add S$. 
\end{theorem}

In the terminology of \cite{rin2} we say that $\Gamma$ and $\Gamma^*$ are \textit{nearly Morita equivalent}. Note that this gives a generalization of the Bernstein-Gelfand-Ponomarev equivalence discussed earlier.   

We have now seen a way of constructing $\modG$ from $\modH$, via $\CH$. There is also
another way, going instead via the tilted algebra $\Lambda = \End_H(T)^\op$, when $T$ is
a tilting $H$-module.

\begin{theorem} \label{3.8}
Let $T$ be a tilting module over the hereditary algebra $H$, and $\Lambda =
\End_H(T)^\op$.

\begin{enumerate}
\item The cluster-tilted algebra $\Gamma = \End_\CH(T)^\op$ is isomorphic to the
trivial extension algebra $\Lambda \ltimes \Ext^2_\Lambda(D\Lambda,\Lambda)$.
\item The quiver of $\Gamma$ is obtained from the quiver with relations of
$\Lambda$, by adding an arrow from $j$ to $i$ for each relation from $i$ to $j$ in a minimal set of relations.
\end{enumerate}
\end{theorem}

\bigskip
\noindent {\bf Example:} If we have the tilted algebra
$$
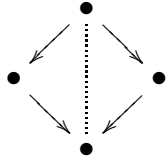
\xymatrix@R0.5cm@C0.5cm{
        & \bullet \ar[dl] \ar[dr] \ar@{.}[dd] &       \\
\bullet \ar[dr]   &                   & \bullet \ar[dl] \\
        & \bullet                 &
}
$$
the associated cluster-tilted algebra has quiver
$$
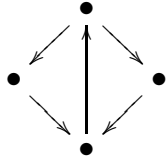
\xymatrix@R0.5cm@C0.5cm{
        & \bullet \ar[dl] \ar[dr] &       \\
\bullet \ar[dr]   &           & \bullet \ar[dl] \\
        & \bullet \ar[uu]     &
}
$$
The same rule applies when we start with a canonical algebra instead of a tilted algebra.

\bigskip
\noindent {\bf Example:} Let $\la$ be a canonical algebra over $k$ given by the quiver
$$\xymatrix@R0.1cm@C0.5cm{
  & \bullet\ar[rr]&& \bullet\ar[ddr]&\\
&&\bullet\ar[drr]&&\\
  a\ar[uur]\ar[urr]\ar[rr]\ar[drdr]\ar[drr]&&\bullet\ar[rr]&& b\\
  &&\bullet \ar[urr]&&\\
  && \bullet \ar[urur] &&
}$$
where $\la=\End_{\coh \mathbb{X}}(T)^{\op}$ for a tilting object $T$ in the associated
category $\coh \mathbb{X}$ of coherent sheaves on a weighted projective line. Then we
obtain the quiver for the algebra $\End_{{\cc}_{\coh\mathbb{X}}}(T)^{\op}$ by adding $5-2=3$
arrows from $b$ to $a$ in the above quiver.
\bigskip

\subsection{Homological properties}
While the tilted algebras have global dimension at most two, it turns out that the
cluster-tilted algebras typically have infinite global dimension. But they have other
homological similarities with hereditary algebras.

Recall that a finite dimensional algebra $\Gamma$ is \emph{Gorenstein} of dimension at
most one if $\Id {_\Lambda \Lambda} \leq 1$ and $\Id \Lambda_\Lambda \leq 1$. (Using
tilting theory, one knows that the last condition can be dropped, see \cite {ar2}). Clearly hereditary
algebras satisfy this property, and we also have the following.

\begin{theorem} \label{3.9}
The cluster-tilted algebras are Gorenstein of dimension at most one.
\end{theorem}

We shall give another homological property of cluster-tilted algebras.
Recall that for a Gorenstein algebra $\Gamma$ of dimension at most one the category
$\Sub\Gamma$ (which is sometimes called the category $\CM\Gamma$ of Cohen-Macaulay modules) 
is functorially finite \cite{as} and extension closed since $\Gamma$ is a cotilting
$\Gamma$-module. Hence $\Sub \Gamma$ has almost split sequences \cite{as}. $\Sub \Gamma$
is also a Frobenius category, that is, $\Sub \Gamma$ has enough projectives and enough
injectives, and the projectives and injectives coincide. Hence the stable category
$\stSubG$ is a triangulated category \cite{ha1}. We have $[1]=\Omega^{-1}$, where $\Omega\colon \ul{\Sub}\Gamma\to \ul{\Sub}\Gamma$ is the equivalence induced by the first syzygy functor.  There is the following necessary
condition on cluster-tilted algebras.

\begin{theorem} \label{3.11}
With the above notation, for a cluster-tilted algebra $\Gamma$, the stable category
$\stSubG$ is $3$-CY, that is $D\Hom_{\stSubG}(A,B)\simeq\Hom_{\stSubG}(B,A[3])$ for $A$, $B$ in
$\stSubG$.
\end{theorem}

We note that there are algebras satisfying both the above homological conditions, but
which are not cluster-tilted.

\bigskip
\noindent {\bf Example:} Let $Q$ be the quiver
$$
\xymatrix{
1 \ar[r]^\alpha & 2 \ar[d]^\beta    \\
4 \ar[u]^\delta & 3 \ar[l]^\gamma }
$$
and let $\Lambda$ be the path algebra $kQ$ modulo the relations given by all paths of
length $7$. This is a Nakayama algebra, which is selfinjective and hence Gorenstein. The
indecomposable projectives have length $7$, and we have $\stSubL = \stmodL$.

As we shall see in \secV, in order to show that $\ul{\mod}\la$ is 3-CY, it is enough to
show that $\tau\simeq\Omega^{-2}$. Since $\Lambda$ is of finite representation type, it is enough to show that $\tau X\simeq \Omega^2(X)$ for each indecomposable $X$ in $\ul{\mod}\Lambda$ \cite{am1}\cite{hj}. We have $\tau(S_1) \simeq S_2$, and $\Omega^{-1}(S_1) = P_3 /
S_1$ and $\Omega^{-1}(P_3 / S_1) = S_2$, so $\tau(S_1) \simeq \Omega^{-2}(S_1)$.
Calculating further, we then see that $\modL$ is $3$-CY. But $\Lambda$ is not
cluster-tilted since the relations in the unique cluster-tilted algebra with this quiver
are paths of length $3$.

\bigskip
\noindent {\bf Notes.} Most of the material in this section is taken from \cite{bmrrt}, \cite{bmr1},
\cite{bmr2}, \cite{bmr3}. Proposition~\ref{3.4} is taken from Assem-Br\"ustle-Schiffler
and Reiten-Todorov (see \cite{bmr3}) and \cite{kr1}, Theorem~\ref{3.8} is taken
from~\cite{abs} (see also \cite{br1}), Theorems ~\ref{3.9} and ~\ref{3.11} from~\cite{kr1},
while for selfinjective cluster tilted algebras the last one is in \cite{gk}. See also
\cite{kz}, \cite{z}, \cite{bkl}, \cite{bv}.

\bigskip
\section{Interplay and applications}
In this section we give some illustration of how the theory of cluster categories and
cluster-tilted algebras has had some feedback on the theory of cluster algebras in the acyclic case, and we
also give examples of nice interplay.

\subsection{Finite mutation classes}
Recall from \secI that there is a finite number of cluster variables, equivalently a
finite number of clusters, equivalently a finite number of seeds, if and only if one of
the seeds contains a Dynkin quiver. However, there may be a finite number of quivers
occurring  even if none of the quivers is Dynkin. Actually we have the following answer 
to a question of Seven.
\begin{theorem}
If the cluster quiver $Q$ has no oriented cycles, then there is only a finite number of
quivers in the mutation class of $Q$ if and only if $Q$ is Dynkin or extended Dynkin or
has two vertices.
\end{theorem}

An essential point to use is that the quivers occurring  are exactly the quivers of
cluster-tilted algebras by Theorem 3.2. To investigate this we use tilting theory for
tame and wild hereditary algebras.

\bigskip
\noindent {\bf Examples:} 1) $\xymatrix@C0.5cm{
  \bullet\ar@<0.5ex>[r]\ar@<-0.5ex>[r]& \bullet
}$ is the only quiver in its mutation class.

\noindent 2) The mutation class of $\xymatrix@R0.5cm@C0.4cm{
  \bullet \ar[r]\ar@/_/[rr]&
  \bullet \ar[r]& \bullet
}$ has in addition only $\xymatrix@R0.5cm@C0.4cm{
  \bullet\ar@/_/[rr]\ar@<-0.4ex>@/_/[rr]&
  \bullet \ar[l]& \bullet\ar[l]
}$

\bigskip
\subsection{Lists of Happel-Vossieck and Seven}
There is a well known \textit{Happel-Vossieck list} in representation theory, which consists
of quivers with relations for the minimal tilted algebras of infinite representation
type, that is, the tilted algebras $\la$ which are of infinite type, but where $\la/\la
e\la$ is of finite type for any vertex $e$ in the quiver \cite{hv}. On the other hand
there is the list of Seven of minimal infinite cluster quivers, namely the quivers not
mutation equivalent to a Dynkin quiver, but if one vertex is removed, the quiver is
mutation equivalent to a Dynkin quiver. As pointed out in \cite{s1}, there is a close
connection between these lists, as the list of Seven is obtained from the Happel-Vossieck
list by inserting arrows in the opposite direction whenever there is a dotted arrow
indicating a minimal relation.

This is explained using that the list of Seven gives the quivers of the minimal
cluster-tilted algebras of infinite type, using Theorem 3.2. Then we also use how to
obtain the quiver of a cluster-tilted algebra from the quiver with relations for the
corresponding tilted algebra, as we have discussed in Theorem 3.2. For example one passes
from
$$
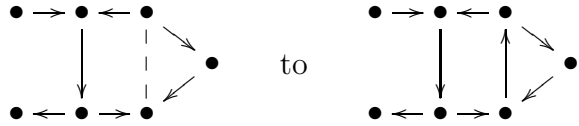
\xymatrix@R0.2cm@C0.4cm{
  \bullet\ar[r]&  \bullet\ar[dd]&  \bullet\ar[l]\ar[dr]\ar@{--}[dd]& &&
  \bullet\ar[r]&\bullet
\ar[dd]&  \bullet\ar[l]\ar[rd]&\\
  &&& \bullet\ar[dl] & \text{ to }&& && \bullet\ar[ld]\\
  \bullet&  \bullet\ar[l]\ar[r]&  \bullet&&&\bullet&  \bullet\ar[l]\ar[r]&  \bullet\ar[uu]&
}$$

\bigskip
\subsection{Denominators}
We have seen that for acyclic cluster algebras there are many similarities between the
ingredients in the definition of a cluster algebra and the cluster-tilting theory in the
corresponding cluster category, for example we have a cluster graph and a cluster-tilting
graph, with seeds or tilting seeds at each vertex.

The next step is to try to define a map from cluster variables to indecomposable
rigid objects, which takes clusters to tilting objects and seeds to tilting seeds
(or in the other direction), which is 1-1 or surjective or both. When  $Q$ is a cluster
quiver without oriented cycles, we have the initial seed $(\ul{x},Q)$, where
$\ul{x}=\{x_1,\cdots,x_n\}$. The natural initial tilting seed is $(H,Q)$, where
$H=kQ=P_1\oplus \cdots\oplus P_n$, and the $P_i$ are the indecomposable projectives. We
can start with defining $\varphi(x_i)=P_i$. When we do the exchange of cluster variables,
we have a corresponding exchange of indecomposable rigid objects, and it is natural to send the
new cluster variable to the new indecomposable rigid object, as illustrated in the following
example, where $Q$ is $1\to 2\to 3$.
$$\xymatrix@C0.1cm{
&\{\{x_1,x_2,x_3\};Q\}\ar@{-}[dl]\ar@{-}[rd]&\\
\{\{{\frac{1+x_2}{x_1}},x_2,x_3\};Q_2\}&& \{\{x_1,{\frac{x_1+x_3}{x_2}},x_3\};Q_3\} }$$
$$\xymatrix@C0.1cm{
& \{P_1\oplus P_2\oplus P_3 ;Q\}\ar@{-}[dr]\ar@{-}[dl]&\\
\{P_1[1]\oplus P_2\oplus P_3;Q_2\} &&\{P_1\oplus S_2\oplus P_3 ;Q_3\} }$$ Here $Q_2$ and
$Q_3$ are as in Section 1.3. We define $\varphi(x_i)=P_i$, $\varphi(\frac{1+x_2}{x_1})=P_1[1]$,
$\varphi(\frac{x_1+x_3}{x_2})=S_3$ etc.

\bigskip
Note that by the theory we have already discussed, we have the same corresponding quiver
in the second diagram, which by definition is the quiver $Q_{T}$ for the cluster-tilting
object $T$.

If we follow the same fixed path from the initial seed in both pictures, the procedure for
defining $\varphi$ is unique. But if we reach the same cluster variable via a different
path, then we might get a different indecomposable object. If there is a map from cluster
variables to indecomposable rigid objects sending clusters to tilting objects and
seeds to tilting seeds, then it has to be given by $\varphi$, so the problem is to prove
that $\varphi$ is well defined.

As we have discussed in Section 1, for our example $Q:1\to 2\to 3$, the denominators of
the cluster variables, when written in reduced form, are given by the composition factors
of an indecomposable rigid module. For hereditary algebras indecomposable rigid modules are uniquely determined by their
composition factors (see \cite{ker1}). And for this example, the map $\varphi$ is well
defined and is a bijection.

Note that if we know that the denominator of any cluster variable is given by the
composition factors of some indecomposable rigid object, and that for any choice of
path from the initial seed the map $\varphi$ takes the cluster variable to this
corresponding rigid object, then the definition of $\varphi$ would not depend on
the choice of paths. Actually, these two statements can be proved simultaneously.
\begin{theorem} Let $Q$ be a cluster quiver without oriented cycles.
\begin{enumerate}
\item[(a)] For each cluster variable $f/g$ different from $x_1,\cdots,x_n$, in reduced form,
there is a unique indecomposable rigid $H$-module whose composition factors are
given by $g$.

\item[(b)] The map $\varphi$ from cluster variables to indecomposable rigid objects
discussed above is well defined and surjective and takes clusters to cluster-tilting objects and
seeds to tilting seeds.
\end{enumerate}
\end{theorem}
One of the problems dealing with cluster variables is to decide when a given expression
$f/g$ is in reduced form. There is a surprisingly elementary positivity condition to deal
with this problem.

We say that $f=f(x_1,\cdots,x_n)$ satisfies the \emph{positivity condition} if $f(e_i)>0$
for $e_i=(1,1\cdots,0,1,\cdots,1)$, where 0 is in the $i$-th position, for
$i=1,\cdots,n$. The following result is crucial.
\begin{lemma}
  If $f$ satisfies the positivity condition and $m$ is a monomial, then $f/m$  is in
reduced form.
\end{lemma}
\begin{proof}
  Assume $f=f_1\cdot x_i$. Then $f(e_i)=f_1(e_i)\cdot0=0$, so that $f$ does not satisfy
the positivity condition. Hence $f/m$ is in reduced form.
\end{proof}
As a consequence of Theorem 4.2, we get the following, answering a conjecture from
Section 1 in the acyclic case.
\begin{theorem}
  For an acyclic cluster algebra, a seed is determined by its cluster.
\end{theorem}

Note that this has been proved in a more general setting in \cite{gsv}.

\smallskip

We illustrate some of these ideas on our standard example.

\bigskip
\noindent {\bf Example:} Let $Q$ be the quiver $1\to 2\to 3$ and $H=kQ$ the corresponding
path algebra. We define a map $\psi$ from the set of isomorphism classes of
indecomposable $H$-modules to the cluster variables for the cluster algebra $C(Q)$, where
the composition factors of $X$ in $\ind H$ determine the denominator in $\psi(X)$. We
here use the AR-quiver, and it is natural to define a map in the opposite direction from
what we considered above. The map $\psi$ is given by the following pictures
$$
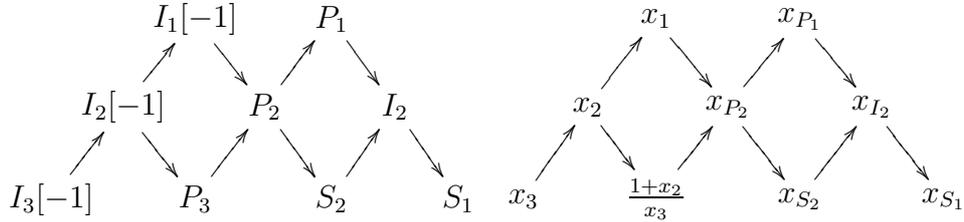
\xymatrix@R0.5cm@C0.2cm{
&&\makebox[5mm]{$I_1[-1]$}\ar[dr]&&P_1\ar[dr]&&&&&x_1\ar[dr]&&\makebox[5mm]{$x_{P_1}$}\ar[dr]&&\\
& \makebox[5mm]{$I_2[-1]$}\ar[ur]\ar[dr]&&P_2\ar[ur]\ar[dr]&&I_2\ar[dr]&&&
x_2\ar[ur]\ar[dr]&&\makebox[5mm]{$x_{P_2}$}\ar[ur]\ar[dr]&&\makebox[5mm]{$x_{I_2}$}\ar[dr]&\\
\makebox[5mm]{$I_3[-1]$}\ar[ur]&&P_3\ar[ur]&& S_2\ar[ur]&&S_1&
x_3\ar[ur]&&\makebox[5mm]{$\frac{1+x_2}{x_3}$}\ar[ur]&&\makebox[5mm]{$x_{S_2}$}\ar[ur]&&\makebox[5mm]{$x_{S_1}$} }
$$
Note that $I_j[-1]=P_j[1]$ in the cluster category $\CC$ and we send $I_j[-1]$ to $x_j$.
Then the cluster-tilting object $T=I_1[-1]\oplus I_2[-1] \oplus I_3[-1]$ is sent to the
cluster $\ul{x}=\{x_1, x_2,x_3\}$. Exchanging $I_2[-1]$ in $T$ gives $S_3=P_3$, and mutating
$\ul{x}$ at vertex 3 amounts to replacing $x_3$ by $x_3^{'}=\frac{1+x_2}{x_3}$. So we
define $\psi(P_3)=\frac{1+x_2}{x_3}$. Denote $\psi(M)=x_M=f_M/m_M$ in reduced form,
where $m_M$ is a monomial. We then get $x_{P_2}=(x_1+\frac{1+x_2}{x_3})\cdot
\frac{1}{x_2}=\frac{1+x_2+x_1x_3}{x_1x_3}$ and
$x_{P_1}=(1+x_{P_2})1/x_1=\frac{1+x_2+x_2x_3+x_1x_3}{x_1x_2x_3}$. We see that the
denominators for the $x_{P_i}$ correspond to the compostion factors of the $P_i$. Further
we have $x_{S_3}=\frac{1+x_{P_2}}{x_{S_3}}=(1+{f_{P_2}/m_{P_2}})\cdot
m_{S_3}/f_{S_3}=\frac{m_{P_2}+f_{P_2}}{f_{S_3}}\cdot\frac{1}{m_{P_2}/m_{S_3}}$. The
monomial in the denominator is $m_{P_2}/m_{S_3}=m_{S_2}$. We do however need that such an
expression is in reduced form, and here the positivity condition is important. We have
that $f_{P_2}$ and $f_{S_3}$ satisfy the positivity condition, and it follows from this
that $\frac{m_{P_2}+f_{P_2}}{f_{S_3}}$ also satisfies the same condition, hence this is
fine. Continuing this ``knitting'' procedure we get a map $\psi$ as indicated on the
picture, where the denominators of the cluster variables correspond to the composition
factors of the $H$-module they come from.

We can use the same procedure for any $H$ of finite type, and we get in this elementary
way a one-one map from indecomposable objects in $\CC$ to cluster variables. To show that
it is a bijection  one can for example use the fact mentioned in \secI that the number of
indecomposable objects in $\CC$ is the same as the number of cluster variables.

In a similar way we define for any $H=kQ$ a one-one map $\psi$ from the indecomposable
preprojective modules to cluster variables. This gives an alternative proof of the fact
that if $Q$ is a connected quiver which is not Dynkin, then there is an infinite number
of cluster variables. Note however that the important Laurent phenomenon is used in all
considerations.

\bigskip
\subsection{Bijection and further applications}
In this section we mention some further improvements, as a consequence of using some more
advanced techniques, in particular a beautiful formula of Caldero-Chapoton involving
Euler characteristics, generalized by Caldero-Keller. This allows one to get a natural
map $\psi$ from indecomposable rigid objects to cluster variables. As a consequence
we have the following.
\begin{theorem}
  The map $\psi$ gives a bijection from indecomposable rigid objects to
cluster variables, taking cluster-tilting  objects to clusters and tilting seeds to
seeds.
\end{theorem}
There are also further results on cluster variables, answering  conjectures from Section
1 in the acyclic case. 
\begin{theorem}
  For any cluster variable $f/m$ in reduced form  for an acyclic cluster algebra,
all coefficients of $f$ are positive.
\end{theorem}
\begin{theorem}
  For an acyclic cluster algebra there is a unique way of replacing a cluster
variable in a cluster by another cluster variable to obtain a cluster.
\end{theorem}

Note that Theorem 4.7  has been proved in a more general setting in \cite{gsv}.
 
There are additional results, so far only proved for finite representation type.
\begin{theorem}
  For any acyclic cluster algebra of finite type, the image under $\psi$ of the
rigid objects in the cluster category give  a $\mathbb{Z}$-basis for the cluster
algebra.
\end{theorem}

\bigskip
\noindent {\bf Notes:} Secton 4.1 is taken from \cite{br2}\cite{s2} (see also \cite{fst}\cite{do}\cite{t}), Section 4.2 from \cite{brs}. For Sections 4.3 and 4.4
see \cite{bmr2}, \cite{bmrt} (with appendix), \cite{cc}\cite{ck1},\cite{ck2}\cite{cr} (see also \cite{bmr4}\cite{bm2}).

\bigskip

\section{2-Calabi Yau categories}

Many of the results on cluster categories and cluster-tilted algebras have a natural
generalization to the more general class of $\Hom$-finite triangulated 2-CY categories
over $k$, with an appropriate choice of special objects and associated algebras. In this
section we give a brief account of this development.

\bigskip

\subsection{Connection with almost split sequences/triangles}
\label{sec_5-1}



 Recall that a $\Hom$-finite triangulated $k$-category $\C$ is \textbf{2-CY}
 if and only if there exists a functorial isomorphism $D\Ext^1_\C(A,B) \simeq \Ext^1_\C(B,A)$ for $A$ and
 $B$ in $\C$. Since the symmetry property for $\Ext^1_\C(\  ,\  )$ plays
 a crucial role in the investigation of cluster catgories, in
 particular the fact that $\Ext^1_\C(A,B)=0$ if and only if
 $\Ext^1_\C(B,A)=0$, it is natural to look for generalizations to
 2-CY categories.

We have that $\cc$ is 2-CY if
and only if $D\Ext^1_{\cc}(A,B)\simeq\Hom_{\C}(B,A[1])$.
 The last formula shows the close connection with $\C$ having almost
 split triangles, and in fact $\C$ being 2-CY is equivalent to $\C$
 having almost split triangles with the corresponding translate
$\tau$ being isomorphic to $[1]$ (see \citeRV).

 The original formulas from which existence of almost split
 sequences was deduced, were valid for module categories or
 subcategories of module categories. In some cases there is however
 a direct reformulation in closely associated triangulated
 categories. For example, let $\Lambda$ be a finite
 dimensional selfinjective algebra. Then the stable category $\stmodL$ of the
 category $\modL$ of finitely generated $\Lambda$-modules is
 known to be triangulated with shift $[1]=\Omega^{-1}$, the first
 inverse syzygy \cite{ha1}. Then we have the following.
 
 \begin{proposition}
 Let $\Lambda$ be a finite dimensional selfinjective algebra. 
 \begin{enumerate}
 \item[(a)]
 In $\ul{\mod}\Lambda$ we have a functorial isomorphism $\ul{\Hom}(B,C)\simeq D\ul{\Hom}(C,\tau \Omega^{-1}B)$ (that is, the functor $\tau \Omega^{-1}\colon \ul{\mod}\Lambda\to \ul{\mod}\Lambda$ is a Serre functor).
 \item[(b)] 
 $\ul{\mod}\Lambda$ is $2$-CY if and only if  $\tau\simeq\Omega^{-1}$ as functors from $\ul{\mod}\Lambda$ to $\ul{\mod}\Lambda$.
 \end{enumerate}
 \end{proposition}

 \begin{proof}
 $(a)$ We have $$\ul{\Hom}(B,C)\simeq D\Ext^1(\tau^{-1}C,B)\simeq D\ul{\Hom}(\tau^{-1}C,\Omega^{-1}B)\simeq  D\ul{\Hom}(C,\tau \Omega^{-1}B).$$
 Here the first isomorphism is the formula on which the existence of almost split sequences is based, and the last two follow directly for selfinjective algebras. 
 
 $(b)$ It follows from $(a)$ that $\ul{\mod}\Lambda$ is $2$-CY if and only if  $\tau\Omega^{-1}\simeq \Omega^{-2}$, if and only if $\tau\simeq\Omega^{-1}$. 
 \end{proof}
 
 Also for a commutative complete local isolated Gorenstein singularity $R$ we have a similar result, for the same reason, since we have a corresponding formula for the category $CM(R)$ of maximal Cohen-Macaulay $R$-modules \cite{au}, and $\ul{CM}(R)$ is triangulated since $CM(R)$ is a Frobenius category.
 
 \begin{proposition}
 Let the notation be as above, with $\dim R=d$.
 \begin{enumerate}
 \item[(a)] We have a functorial isomorphism $$\ul{\Hom}(B,C)\simeq D\ul{\Hom}(C,\tau\Omega^{-1}B)\simeq D\ul{\Hom}(C,\Omega^{1-d}B).$$
 \item[(b)] $\ul{CM}(R)$ is $(d-1)$-CY, in particular $2$-CY if $d=3$. 
 \end{enumerate}
 \end{proposition}
 
 Note that the formula $\tau\simeq\Omega^{2-d}$ is given in \cite{au}.
\bigskip

\subsection{Cluster-tilting objects}
\label{sec_5-2}

 For cluster categories we have considered  the concepts of maximal rigid objects and cluster-tilting objects, and we have seen that they coincide. In the more general context of triangulated $2$-CY categories this is not the case \cite{bikr}. It turns out that cluster-tilting is the natural condition to use in general since the extra property required here is essential in some of the proofs.
 
 The algebras $\End_\C(T)^{\op}$ for $T$ a cluster-tilting object in a $\Hom$-finite triangulated $2$-CY category $\C$ are called \textit{$2$-CY-tilted algebras}. This class properly contains the class of cluster-tilted algebras.
 
 In a triangulated 2-CY category $\C$ it may happen that there is no cluster-tilting object, actually even no nonzero object $M$ with $\Ext^1_\C(M,M)=0 $. Sometimes there is  instead what is called a cluster-tilting
 subcategory, which may have an infinite number of indecomposable
 objects. Here one requires in addition that the category is
 functorially finite, as done for maximal 1-orthogonal
 subcategories, but which was not required for $\Ext$-configurations
 (see \cite{kr1}, \cite{birsc} for such examples).

\bigskip
\subsection{Analogous results}
\label{sec_5-3}

Much of the general theory in Sections 3 and 4 carries over to the setting of triangulated $2$-CY categories and $2$-CY-tilted algebras. It is however not the case in general that the quivers of the $2$-CY-tilted algebras have no loops and $2$-cycles \cite{bikr}, so in order to get cluster quivers we must exclude this possibility. Excluding this, the $2$-CY category with the cluster-tilting objects turns out to have what is called a \textit{cluster structure} \cite{birsc}, which essentially means that all the essential ingredients for having possible connections with cluster algebras hold. Also note that there is no known analog of the description of cluster-tilted algebras as trivial extensions of tilted algebras. 

\bigskip
\subsection{Preprojective algebras of Dynkin quivers}
\label{sec_5-4}

 Important examples of 2-CY categories are the stable module categories
$\ul{\mod}\la$, where $\la$ is a preprojective
 algebra of a Dynkin quiver over a field $k$. Recall that if for
 example $Q$ is the quiver
 
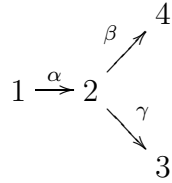
\[ \xymatrix@C0.5cm@R0.5cm{
                    &              & 4 \\
   1 \ar[r]^\alpha  & 2 \ar[ur]^\beta \ar[dr]^\gamma &  \\
                    &              & 3 }
\]
then the preprojective algebra $\Pi(Q)$ is given by the quiver
 
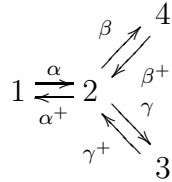
\[ \xymatrix@C0.5cm@R0.5cm{
                    &              & 4 \ar@<0.5ex>[dl]^{\beta^{+}} \\
     1 \ar@<0.5ex>[r]^\alpha
   & 2 \ar@<0.5ex>[l]^{\alpha^{+}} \ar@<0.5ex>[ur]^\beta
       \ar@<0.5ex>[dr]^\gamma  &  \\
                    &              & 3 \ar@<0.5ex>[ul]^{\gamma^{+}} }
\]
with the relations $\alpha\alpha^{+} - \alpha^{+}\alpha + \beta\beta^{+} - \beta^{+}\beta
+ \gamma\gamma^{+} - \gamma^{+}\gamma=0$.

For the case of Dynkin quivers, the preprojective algebras are known to be finite
dimensional selfinjective, and we have the following.
\begin{proposition} \label{prop_5-3}
 When $Q$ is Dynkin, with associated preprojective algebra
 $\Lambda=\Pi(Q)$, then the stable category \emph{$\stmodL$} is $2$-CY.
\end{proposition}
\begin{proof}
The algebra $\Lambda$ is known to be selfinjective.
    In view of Proposition 5.1, we only need to see that
    $\tau\simeq\Omega^{-1}$. This follows from \citeAR. We here give an
    outline of the proof, specialized to the case of interest here.
    The proof is based on some facts about the category $\CM(R)$ of
    maximal Cohen-Macaulay modules over two-dimensional simple
    hypersurface singularities $R$, which are of finite
    Cohen-Macaulay type and correspond to Dynkin diagrams. We have
    that $\tau_R$ is the identity, and $\Omega^2 \simeq id$ on the
    stable category $\uCM(R)$, and we have the formula $D\Ext^1_R(A,B)\simeq \ul{\Hom}_R(B,\tau
    A)$ \cite{au}. In addition, if $M$ is the direct sum of one copy of
    each indecomposable object in $\uCM(R)$ up to isomorphism, then
    $\Gamma = \underline{\End}{_R}(M)^{\op}$ is isomorphic to $\Pi(Q)$, 
    where the underlying graph of $Q$ is the Dynkin diagram
    corresponding to $R$. We view the category $\C=\modG$ as
    $\mod( \uCM(R) )$, the category of finitely presented
    contravariant functors from $\uCM(R)$ to $\mod k$. Denote by
    $\nu_\C = D(\ul{\Hom}_R(\ ,C)^*)$ the Nakayama
    functor, where $C$ is in $\uCM(R)$. Then we have $\tau_\C = \Omega^2_\C
    \nu_\C$. Since $\Gamma$ is selfinjective, we have $D\ul{\Hom}_R(\phantom{-},C)\simeq
    D\Ext^1_R(\Omega^1_R C,\phantom{-} )$, which is isomorphic to $\uHom_R(\Omega^1_R
    C,\phantom{-})$ since $\tau_R = id$. Since
    $\ul{\Hom}_R(Y,\phantom{-})^*=\ul{\Hom}_R(\phantom{-},Y)$ for
    $Y$ in $\CM(R)$, we have $\nu^{-1}_{\C}\ul{\Hom}_R(\phantom{-},C) =
(D\ul{\Hom}_R(\phantom{-},C))^* \simeq \ul{\Hom}_R(\phantom{-}, \Omega^1_R  C)$, hence
$\nu_{\C}\ul{\Hom}_R(\phantom{-},C)=\ul{\Hom}_R(\phantom{-},\Omega^{-1}_RC)$, for $C$ in $\CM(R)$. One can show that $\Omega^{-1}_R \colon \uCM(R) \to
    \uCM(R)$ induces in a natural way a functor $\alpha$ from $\C$
    to $\C$, hence from $\underline{\C}$ to $\underline{\C}$, which
    is isomorphic to $\Omega^{-3}_\C$. It follows that $\tau_\C = \Omega^2_\C
    \nu_\C$ is isomorphic to $\Omega^{-1}_\C$ as functors  from $\underline{\C}$ to
    $\underline{\C}$.
\end{proof}

This case of preprojective algebras has been  investigated extensively in a series of
papers by  Geiss-Leclerc-Schr\"oer. They work in the category $\modL$, rather that in
the 2-CY category $\stmodL$, but the categories $\modL$ and $\stmodL$ are closely
related, and one can go back and forth between exact sequences and triangles.

 As for cluster categories, the concepts of cluster-tilting (maximal 1-orthogonal) and maximal rigid coincide. And
 also one has that the associated 2-CY tilted algebras have no loops
 or 2-cycles in their quiver. In this case there are many
 interesting connections with cluster algebras and Lusztig's dual semicanonical basis
  \cite{gls2}, \cite{gls1}.

\bigskip
\subsection{Further examples}
\label{sec_5-5}

 We have already indicated that examples of 2-CY
categories may be found amongst the stable categories $\uCM(R)$, where $R$ is a complete
local commutative noetherian isolated Gorenstein singularity. In view of
Proposition 5.1, all we have to check is that we have an isomorphism of functors $\tau \simeq\Omega^{-1}$ from $\uCM(R)$ to $\uCM(R)$. As we have seen it is
known from the work of Auslander \cite{au}\ that if $d=\dim R$, then $\tau\simeq
\Omega^{2-d}$, hence $\uCM(R)$ is $2$-CY if $d=3$.

A concrete example is the following: Let $S= k[[X,Y,Z]]$ and let $G$ be the subgroup
$\left\langle\left( \begin{smallmatrix} \xi & 0 & 0 \\ 0 & \xi & 0 \\
0 & 0 & \xi \end{smallmatrix} \right) \right\rangle$ of the special linear group
$\operatorname{SL}(3,k)$ where $\xi$ is a primitive third root of 1. Then the invariant
ring $R = S^G$ is a 3-dimensional ring with the desired properties, so that $\uCM(R)$ is
$2$-CY.

There is a large class of 2-CY categories associated with preprojective algebras $\Lambda$ of
quivers which are not Dynkin \cite{birsc} (see also \cite{gls3}). They arise from taking stable categories of
appropriate subcategories of $\modL$. They contain both the cluster categories and
the stable categories $\stmodL$ where $\Lambda$ is the preprojective algebra of a Dynkin
quiver as special cases.

\bigskip
\subsection{Recognizing cluster categories}
\label{sec_5-6}

 A natural question is whether one can tell from a 2-CY-tilted
 algebra which 2-CY category it came from. In particular, we can
 tell when it comes from a cluster category under some mild assumptions. A triangulated $2$-CY category is \textit{algebraic} if it is the stable category  $\ul{\mathcal{E}}$ of an exact Frobenius category $\mathcal{E}$.
 \begin{theorem} \label{thm_5-4}
  Let $\Gamma$ be a $2$-CY-tilted algebra coming from an algebraic $2$-CY category, whose quiver $Q$ has no
  oriented cycles. Then $\Gamma$ comes from a cluster category
  $\C_Q$, and is hence cluster-tilted. 
 \end{theorem}

We illustrate with the example from Section \ref{sec_5-5}. Here $S$ turns out to be a
cluster-tilting object \citeI, and the quiver of $\underline{\End}_R(S)^{\op}$ turns out to
be $\xymatrix@1{ \bullet\ar@<1ex>[r] \ar@<-1ex>[r] \ar[r] & \bullet }$. Hence $\uCM(R)$
is equivalent to $\C_{k \left(  \xymatrix@1@C=1pc{\bullet \ar@<1ex>[r]\ar@<-1ex>[r] \ar[r] & \bullet}
\right)} $.

\bigskip
\noindent {\bf Notes.}
 The generalization from cluster categories to 2-CY categories in
 \ref{sec_5-2} and \ref{sec_5-3} is taken from \cite{kr1}, and the recognition
theorem in
 \ref{sec_5-6} is given in \cite{kr2}. For further work see \cite{iy}\cite{birsc}\cite{birsm}\cite{bikr}\cite{dk}\cite{fk}\cite{am1}\cite{am2}\cite{p1}\cite{p2}.

There has also recently been work devoted to higher cluster categories \\${\bf D}^b(H)/\tau^{-1}[d-1]$, and more generally triangulated $d$-Calabi-Yau categories. But we will not discuss these aspects in this chapter.

\end{document}